\documentclass[reqno,11pt]{amsart}
\usepackage{latexsym}
\usepackage{amsthm,amsmath,amssymb}
\usepackage[english]{babel}
\usepackage{enumerate}
\usepackage{ dsfont}
\usepackage{graphicx}
\usepackage{color}
\usepackage{fullpage}
\usepackage{hyperref}
\usepackage{subfigure}
\usepackage{array}
\newcommand{\real}{\mathbb{R}}

\newcommand{\abs}[1]{\left|#1\right|}

\newcommand{\commentout}[1]{}

\newcommand{\der}[2]{ \frac{\partial #1}{\partial #2} }
\newcommand{\prob}{\mathcal{P}}

\newcommand{\B}{\mathcal{B}}
\newcommand{\C}{\mathcal{C}}
\newcommand{\M}{\mathcal{M}}

\newcommand{\R}{\mathbb{R}}
\newcommand{\N}{\mathbb{N}}

\newcommand{\supp}{\text{\rm supp}}
\def\Xint#1{\mathchoice
{\XXint\displaystyle\textstyle{#1}}%
{\XXint\textstyle\scriptstyle{#1}}%
{\XXint\scriptstyle\scriptscriptstyle{#1}}%
{\XXint\scriptscriptstyle\scriptscriptstyle{#1}}%
\!\int}
\def\XXint#1#2#3{{\setbox0=\hbox{$#1{#2#3}{\int}$ }
\vcenter{\hbox{$#2#3$ }}\kern-.6\wd0}}

\def\dashint{\Xint-}

\newtheorem{lemma}{Lemma}
\newtheorem{corollary}{Corolary}
\newtheorem{definition}{Definition}
\newtheorem{proposition}{Proposition}
\newtheorem{theorem}{Theorem}
\newtheorem{remark}{Remark}

\begin{document}
\title{Dimensionality of Local Minimizers \\ of the Interaction Energy}

\author{D. Balagu\'e $^1$, J. A. Carrillo$^2$,  T. Laurent$^3$, and G. Raoul$^4$ }

\address{$^1$ Departament de
Matem\`a\-ti\-ques, Universitat Aut\`onoma de Barcelona, E-08193
Bellaterra, Spain. E-mail: {\tt dbalague@mat.uab.cat}.}

\address{$^2$ Department of Mathematics, Imperial College London,
South Kensington Campus, London SW7 2AZ, UK. E-mail: {\tt
carrillo@imperial.ac.uk}}

\address{$^3$ Department of Mathematics, University of California -
Riverside, Riverside, CA 92521,  USA. E-mail: {\tt
laurent@math.ucr.edu}.}

\address{$^4$Centre d'Ecologie Fonctionnelle et Evolutive, UMR 5175, CNRS, \\ 1919 Route de Mende, 34293 Montpellier Cedex 5, France. E-mail: {\tt
raoul@cefe.cnrs.fr}.}

\maketitle

\begin{abstract}
In this work we consider local minimizers (in the topology of
transport distances) of the interaction energy associated to
a repulsive-attractive potential.  We show how the dimensionality
of the support of local minimizers is related to the repulsive
strength of the potential at the origin.
\end{abstract}

\section{Introduction}
Given a Borel measurable function $W: \real^N \to
(-\infty,+\infty]$ which is bounded from below, the interaction
energy of the Borel probability measure $\mu$ is given by
\begin{align}
    E_W[\mu]&:=\frac{1}{2}\iint_{\R^N\times\R^N} W(x-y)\,d\mu(x)d\mu(y) \,. \label{defE}
\end{align}
Our main goal will be to analyse the qualitative properties of
local minimizers of the  energy $E_W$ in the set of Borel
probability measures with the topology induced by transport
distances. More specifically, we will show that the Hausdorff
dimension of the support of local minimizers is directly related
to the behavior at the origin of $\Delta W$.

The interaction energy $E_W$ arises in  many contexts. In
physical, biological, and material sciences it is used to model
particles or individuals effects on others via pairwise
interactions. Given $n$ particles located at $X_1, \ldots, X_n \in
\real^N$, their discrete interaction energy is given by
\begin{align}
    E^n_W[X_1, \ldots, X_n ]&:= \frac{1}{2n^2} \sum_{\substack{i,j=1\\j\neq i}}^{n}   W(X_i-X_j). \label{defEdiscrete}
\end{align}
Note that formally for a large number of particles, the discrete
energy \eqref{defEdiscrete} is well approximated by the continuum
energy \eqref{defE} where $d\mu(x)$ is a general distribution of
particles at location $x\in \real^N$. In fact, the continuum
energy \eqref{defE} of the discrete distribution $\frac1n
\sum_{i=1}^n \delta_{X_i}$ reduces to \eqref{defEdiscrete}.

In  models arising in material sciences \cite{Wales1995,
Wales2010,Rechtsman2010, Viral_Capside, soccerball}, particles,
nano-particles, or molecules self-assemble in a way to minimize
energies similar to $E^n_W$. Analogously in applications to
biological sciences \cite{Mogilner2003,Mogilner2,BT1,BT2,
Dorsogna}, individuals in a social aggregate (e.g., swarm, flock,
school, or herd) self-organize in order to minimize similar type
of energies. In these applications the potential $W$ is typically
repulsive in the short range so that particles/individuals do not
collide, and attractive in the long range so that the
particles/individuals gather to form a group or a structure.
Therefore one is often led to consider radially symmetric
interaction potentials of the form $W(x)=w(|x|)$ where $w:
[0,+\infty) \to (-\infty,+\infty]$ is decreasing on some interval
$[0,r_0)$ and increasing on $(r_0,+\infty)$. The function $w$ may
or may not have a singularity at $r=0$.  We will refer to such
potentials as being {\bf repulsive-attractive}. Since $w$ has a
global minimum at $r_0$, it is obvious that if we consider only
two particles $X_1$ and $X_2$, in order to minimize
$E_W^2[X_1,X_2]$, the two particles must be located at a distance
$r_0$ from one another. Whereas the situation is simple with two
particles, it becomes very complicated for large number of
particles. Recent works
\cite{KSUB,BUKB,FHK,soccerball,KHP,HuiUminskyBertozzi,Raoul,
FellnerRaoul1, FellnerRaoul2,BCLR} have shown that such
repulsive-attractive potentials lead to the emergence of
surprisingly rich geometric structures. The goal of the present
paper is to understand how the dimensionality of these structures
depends on the singularity of $\Delta W$ at the origin.

Let us describe the main results. Consider a repulsive-attractive
potential $W(x)=w(|x|)$. Typically the Laplacian of such potential
will be negative in a neighborhood of the origin. We show that if
\begin{equation}\label{paris}
\Delta W(x) \sim -\frac{1}{|x|^{\beta}} \qquad \text{as } x\to 0
\end{equation}
for some $0<\beta < N$, then the support of local minimizers of
$E_W$ has Hausdorff dimension greater or equal to $\beta$. The
precise hypotheses needed on $W$ for this result to be true, as
well as the precise meaning of \eqref{paris}, can be found in the
statement of Theorem \ref{dimension}. The exponent $\beta$
appearing in \eqref{paris} quantifies how repulsive the potential
is at the origin. Therefore our result can be intuitively
understood as follows: the more repulsive the potential is at the
origin, the higher the dimension of local minimizers will be.

Potentials satisfying \eqref{paris} have a singular Laplacian at 0
and we refer to them as {\bf strongly repulsive at the origin}.
The second main result is devoted to potentials which are {\bf
mildly repulsive at the origin}, that is potentials whose
Laplacian does not blow up at the origin. To be more precise we
show that if
\begin{equation}\label{marseille}
W(x) \sim -|x|^{\alpha} \qquad \text{as } x\to 0  \qquad \text{ for  some $\alpha>2$}
\end{equation}
then a local minimizer of the interaction energy cannot be
concentrated on smooth manifolds of any dimension except
0-dimensional sets. The exact hypotheses on $W$, as well as the
precise meaning of \eqref{marseille}, can be found in Theorem
\ref{thm1}. Note that this result suggests that local minimizers
of the interaction energy of mildly repulsive potentials have zero
Hausdorff dimension -- however we are currently unable to prove
this stronger result.

Summarizing, in this paper we show that if the Laplacian of the
potential behaves like $-1/|x|^\beta$ around the origin, with
$0<\beta<N$, then the dimension of minimizers is at least $\beta$
and if the Laplacian does not blow up at the origin, then the
dimension is zero, see the precise statement in Theorems
\ref{dimension} and \ref{thm1}. This is illustrated in the case of
two dimensions ($N=2$) in Table~\ref{figure_intro}, where we show
some local minimizers of $E_W$ with interaction potentials of the
form
\begin{equation} \label{powerlaw}
W(x)=-\frac{|x|^\alpha}\alpha+\frac{|x|^\gamma}\gamma \qquad \alpha<\gamma,
\end{equation}
so that  $W(x)\sim -\frac{|x|^\alpha}\alpha$ and $ \Delta W(x)\sim
-\frac1{|x|^{\beta}}$ with $\beta=2-\alpha$ as $x\to 0$.

\begin{itemize}
\item Subfigure (a): $\alpha=2.5$ and $\gamma=15$. The support of
the minimizer has zero Hausdorff dimension in agreement with
Theorem \ref{thm1}. Actually, in this particular case it is
supported just on 3 points.

\item Subfigure (b) and (c):  we consider two examples where the
potentials have the same behavior at the origin, $\alpha=1.5$, but
different attractive long range behavior ($\gamma=7$ and $2$
respectively). Theorem \ref{dimension} shows that the Hausdorff
dimension of the support must be greater or equal to
$\beta=2-\alpha=0.5$. Indeed, the minimizer for the first example
has a one-dimensional support on three curves whereas the
minimizer for the second example has two-dimensional support.

\item Subfigure (d): $\alpha=0.5$ and $\gamma=5$.  Theorem
\ref{dimension} proves that the Hausdorff dimension of the support
must be greater than $\beta=2-\alpha=1.5$. The numerical
simulation demonstrates that it has dimension two.
\end{itemize}
In our extensive numerical experiments using gradient descent
methods we never observed minimizers with a support that might be
of non-integer Hausdorff dimension.

\begin{table}
\centering
\begin{tabular}{||m{1.7cm}|m{3.5cm}|m{3.5cm}|m{3.5cm}||}
\hline
& $\qquad\mbox{Dim = 0}$ & $\qquad\mbox{Dim = 1}$ & $\qquad\mbox{Dim = 2}$ \\
\hline
& (a)  & & \\
$\alpha=2.5 $
&\begin{center}\includegraphics[height=0.12\textheight]{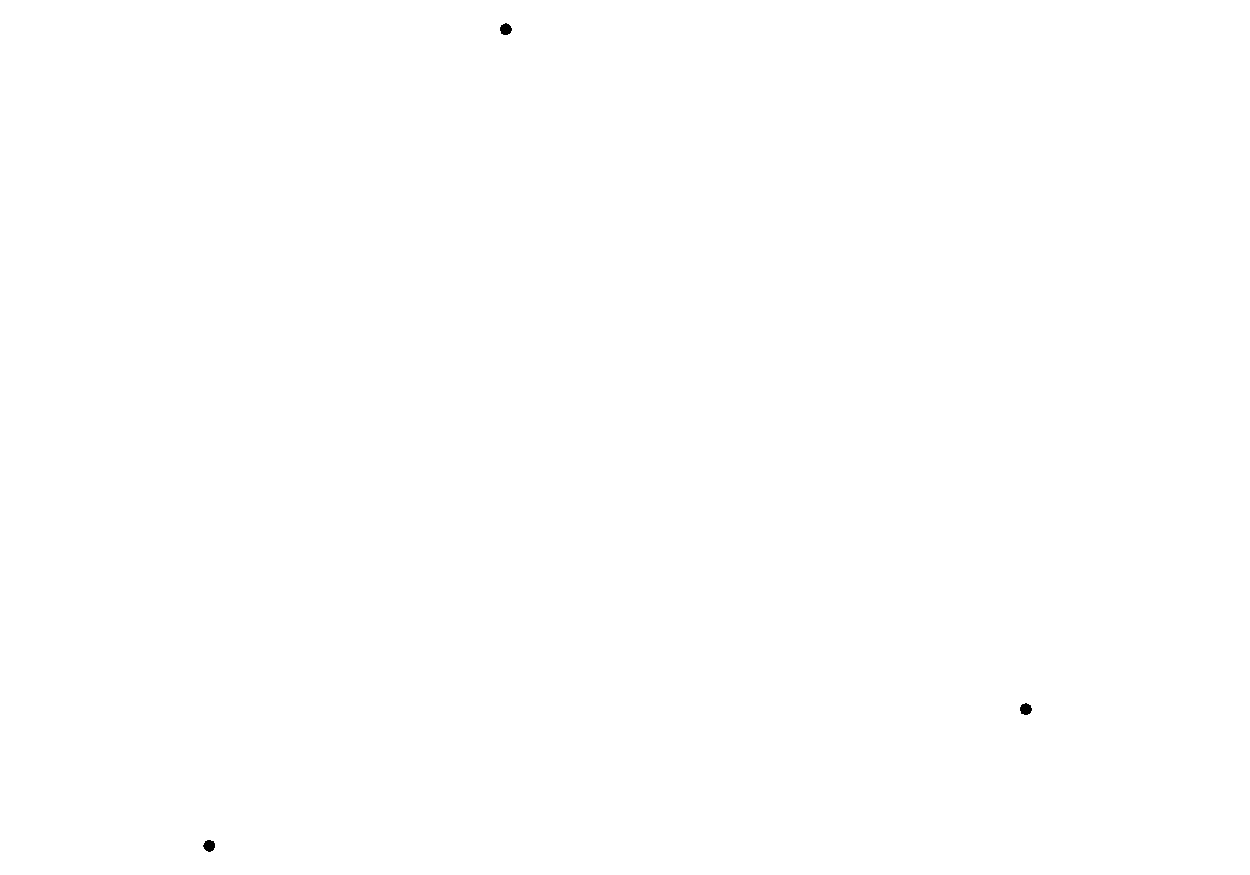}\end{center}
& & \\ \hline
& & (b) & (c) \\
$\alpha=1.5$ &
&\begin{center}\includegraphics[height=0.12\textheight]{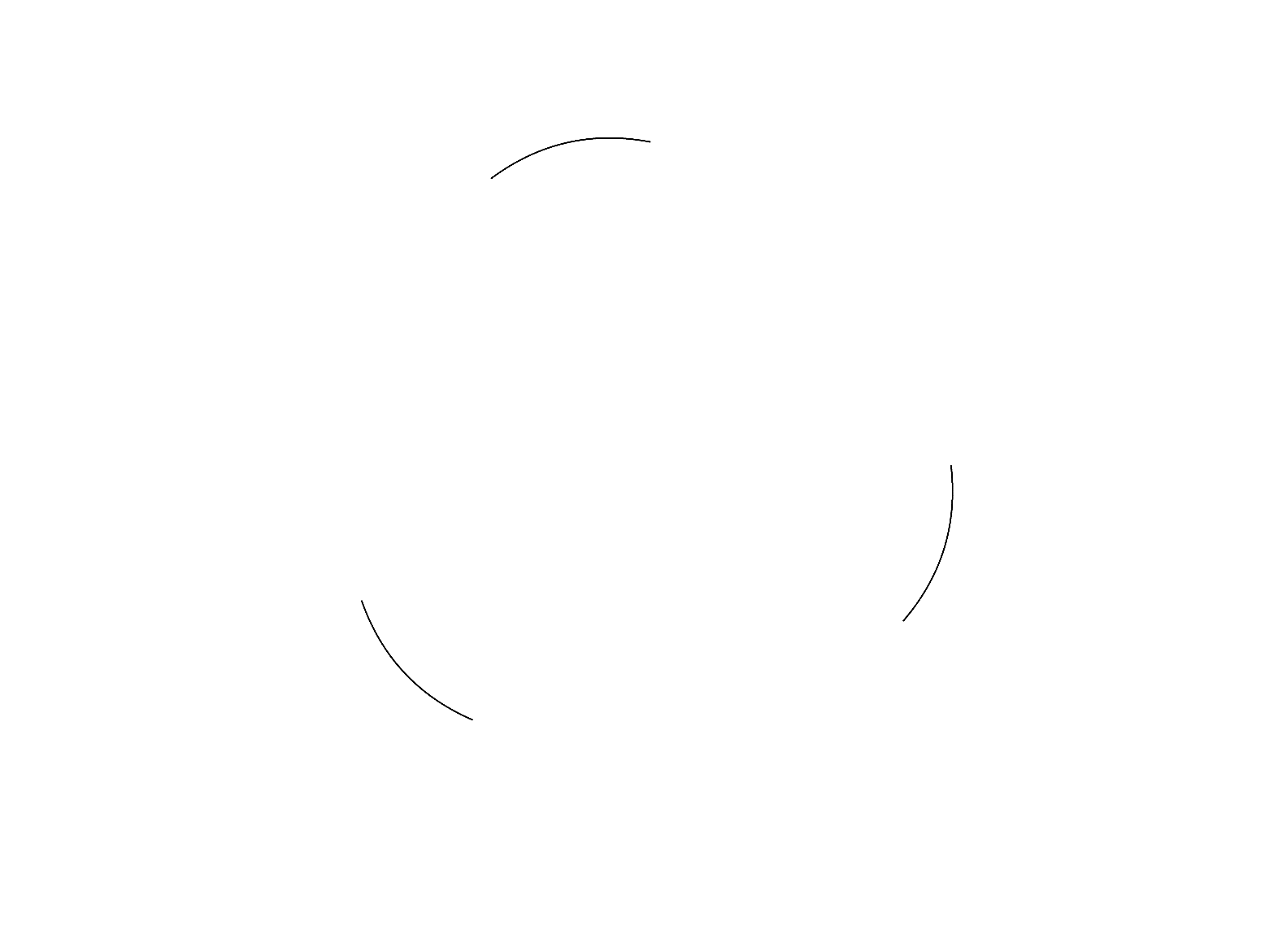}\end{center}
&\begin{center}\includegraphics[height=0.12\textheight]{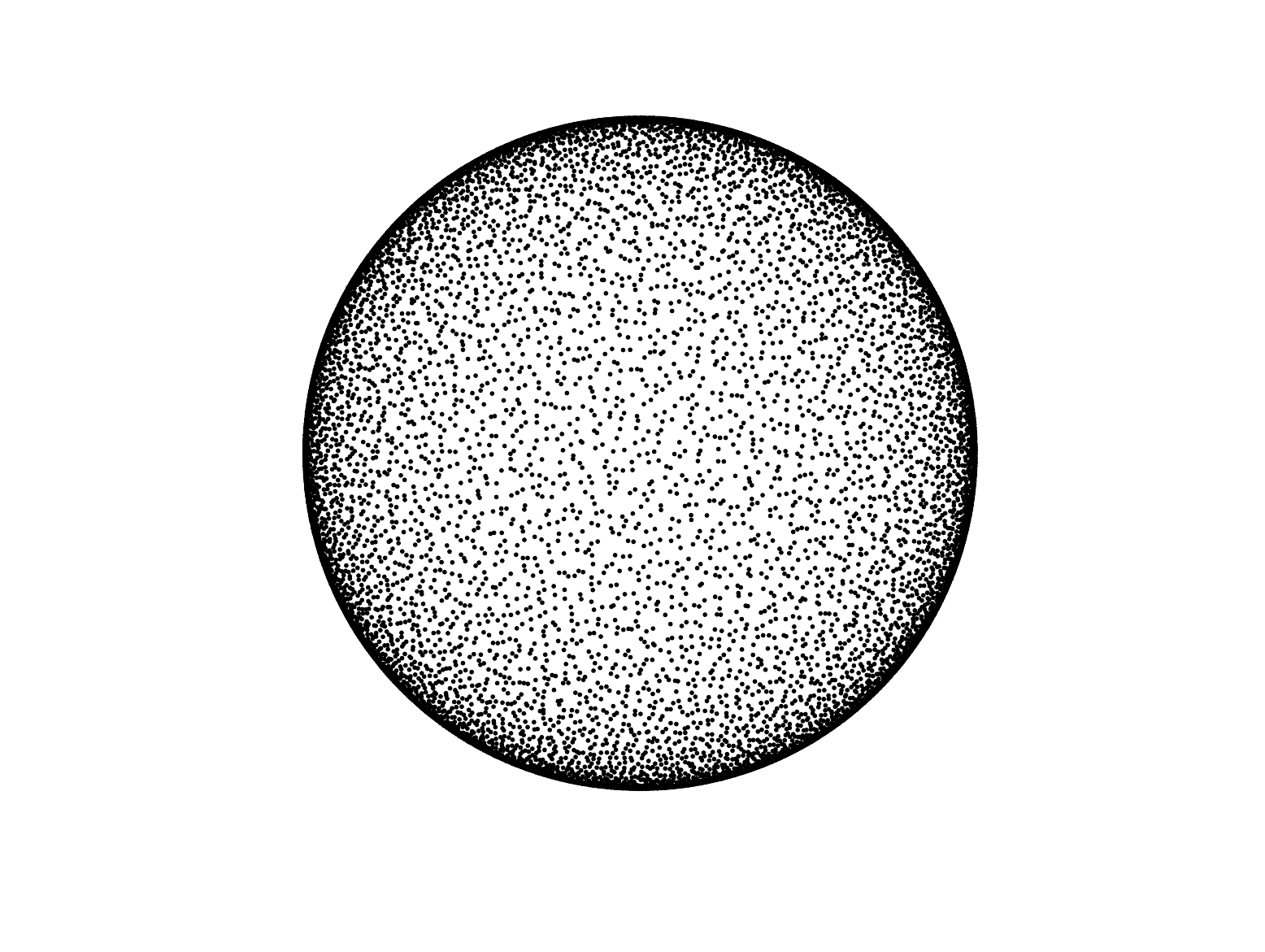}\end{center}
\\ \hline
& & & (d)\\
$\alpha=0.5$ & &
&\begin{center}\includegraphics[height=0.12\textheight]{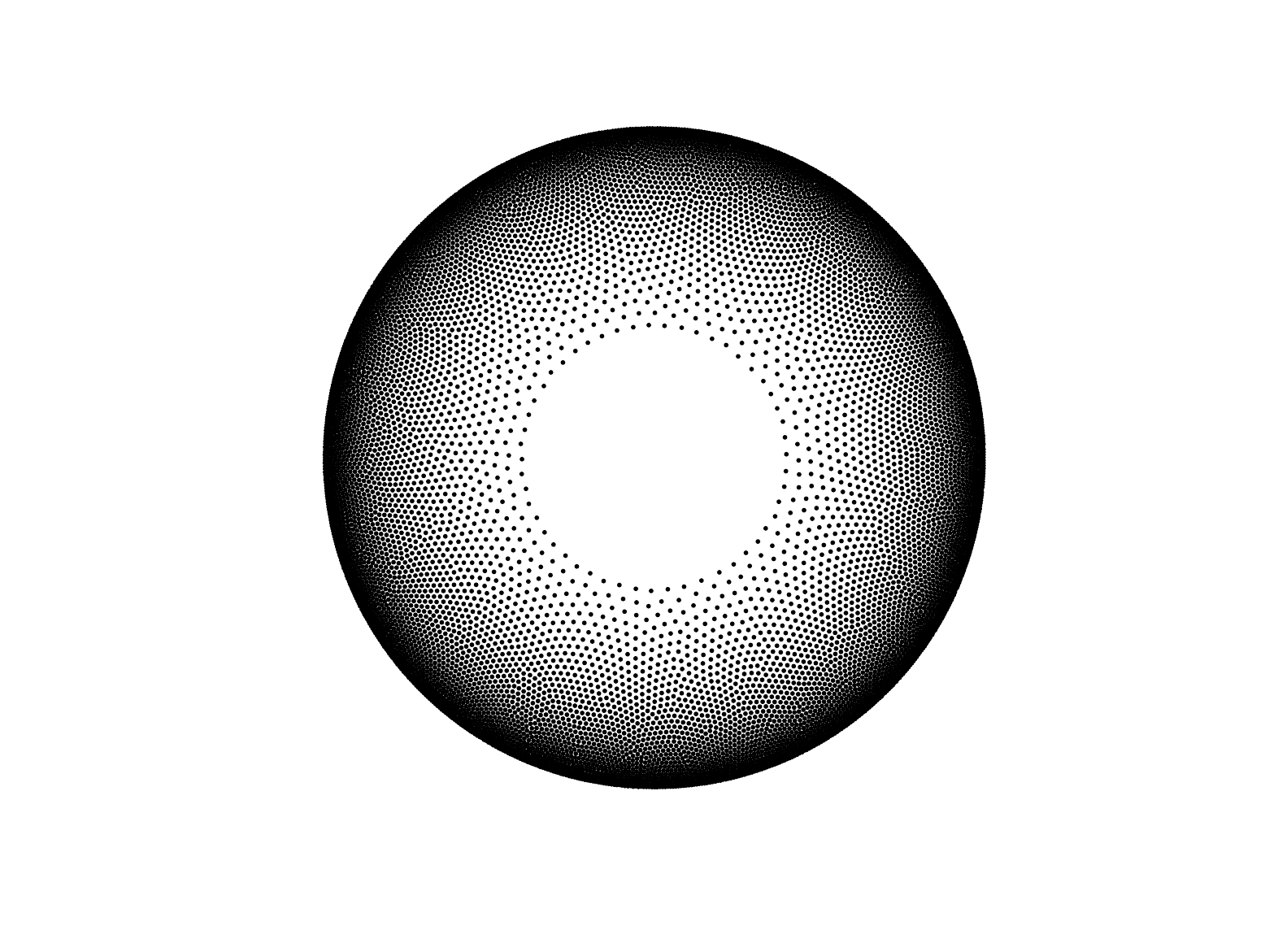}\end{center}
\\ \hline
\end{tabular}
\vspace{3mm} \caption{Local minimizers of the interaction energy
$E^n_W$ for various potentials $W(x)$. In these computations
{$n=10,000$}. When $\Delta W$ does not blow-up at the
origin (Case a) the Hausdorff dimension of the the support of
minimizers is zero. When $\Delta W \sim -1/|x|^{\beta}$ as $x \to
0$, $0<\beta<N$ (Cases b,c,d) the Hausdorff dimension of the the
support of minimizers is greater or equal to $\beta$.}
\label{figure_intro}
\end{table}

In most of this paper, we will consider local minimizers for the
topology induced by the transport distance $d_\infty$ (see section
2 for a definition of $d_\infty$). This topology is indeed the
natural one to consider. In particular, gradient descent numerical
methods based on particles typically lead to local minimizers for
the $d_\infty-$topology. Moreover  the topology induced by
$d_\infty$ is the finest topology among the ones induced by $d_p$,
$1\leq p\leq \infty$ (see section 2 for a definition of $d_p$). As
a consequence local minimizers in the $d_p$-topology are
automatically local minimizers in the $d_\infty$-topology, and
thus they are also covered by our study. In Section 5 we will
discuss in more detail these questions.

Let us finally mention that the gradient flow of the energy $E_W$
in the Wasserstein sense $d_2$
\cite{Carrillo-McCann-Villani03,Ambrosio2008,Carrillo-McCann-Villani06}
has been extensively studied in recent years \cite{L,BL,BCL,
BGL,BLR,CDFLS,CDFLS2,BLL,FellnerRaoul1,FellnerRaoul2,Raoul,BCLR,BCY}.
It leads to the nonlocal interaction equation
\begin{gather}
\der{\mu}{t} + \text{div}(\mu v) =0  \quad , \quad  v= - \nabla W*
\mu   \label{pdes1}
\end{gather}
where $\mu(t,x)=\mu_t(x)$ is the probability or mass density of
particles at time $t$ and at location $x \in \real^N$, and
$v(t,x)$ is the velocity of the particles. Stability properties of
steady states for \eqref{pdes1} with repulsive-attractive
potentials have only been analyzed very recently. In \cite{BCLR}
we gave conditions for radial stability/instability of particular
local minimizers. We should also mention that the one dimensional
case was analyzed in detail in \cite{FellnerRaoul1,FellnerRaoul2}.
Well-posedness theories for these repulsive-attractive potentials
in various functional settings have been provided in
\cite{L,Ambrosio2008,BLR,CDFLS,BLL,BCLR}. Stable steady states of
\eqref{pdes1} under certain set of perturbations are expected to
be local minimizers of the energy functional \eqref{defE} in a
topology to be specified. Actually, this topology should determine
the set of admissible perturbations. As already mentioned, the
$d_\infty$-stability is the one typically studied by performing
equal mass particles simulations.

Finally, we can now interpret our dimensionality result in terms
of the nonlocal evolution equation \eqref{pdes1}. The heuristic
idea behind the implication: \eqref{paris} with $0<\beta<N$
implies dimensionality larger than $\beta$ of the support of local
minimizers of $E_W$; can be understood in terms of the divergence
of the velocity field in \eqref{pdes1}. In fact, it is
straightforward to check that the divergence of the velocity field
generated by a uniform density localized over a smooth manifold of
dimension $k$ is $+\infty$ on the manifold if and only if
$k<\beta$ (this is equivalent to non-integrability of $-\Delta W$
on manifolds of dimension $k$). Heuristically, if $\mbox{div
}v=-\Delta W * \mu$ associated to $\mu$ diverges on its support
the density has a strong tendency to spread, the configuration is
not stable and then $\mu$ is not a local minimizer. Therefore, we
can reinterpret our result in Theorem 1 as follows: local
minimizers of \eqref{defE} have to be supported on manifolds where
the divergence of their generated velocity field is not $+\infty$.

The plan of the paper is as follows. Section 2 will be devoted to
the necessary background in optimal transport theory and
notations. Strongly repulsive potentials are treated in Section 3
while mildly repulsive potentials are analyzed in Section 4. In
Section 5, for the smaller subset of local minimizers in the
$d_2$-topology, we show that we can use an Euler-Lagrange approach
in the spirit of \cite{BT2} to derive some properties of these
minimizers. Extensive numerical tests as well as details of the
algorithm used in order to minimize $E_W^n$ are reported in
Section 6.


\section{Preliminaries in Transport Distances}

We denote by $\B(\real^N)$ the family of Borel subsets of
$\real^N$.  Given a set $A \in \B(\real^N)$, its Lebesgue measure
is denoted by $|A|$. We denote  by $\M(\real^N)$   the set of
(nonnegative) Borel measures on $\real^N$ and by $\prob(\real^N)$
the set of Borel probability measures on $\real^N$. The support of
$\mu \in \M(\real^N)$, denoted by $\supp (\mu)$, is the closed set
defined by
\begin{equation*} \label{supp}
\supp (\mu):= \{x \in \real^N: \mu(B(x,\epsilon))>0  \text{ for
all } \epsilon>0 \}\,.
\end{equation*}
A measure $\rho \in \M(\real^N)$ is said to be a part of $\mu$ if
$\rho(A) \le \mu(A)$ for all $A\in\B(\real^N)$ and it is not
identically zero. This terminology is justified by the fact that
if $\rho$ is a part of $\mu$, then $\mu$ can be written $\mu=
\rho+\nu$ for some $\nu \in \M(\real^N)$ ($\nu=\mu-\rho$ to be
more precise). We will say that a probability measure
$\mu\in\prob(\R^N)$ can be decomposed as a convex combination of
$\mu_0,\mu_1 \in \prob(\R^N)$ if there exists $0\leq m_0,m_1 \leq
1$ with $m_0+m_1=1$ such that $\mu=m_0 \mu_0 + m_1 \mu_1$.

Let us introduce some notation related to the interaction
potential energy. We denote by $B_W:
\prob(\real^N)\times\prob(\real^N) \to (0,+\infty]$ the bilinear
form defined by
\begin{equation}\label{defB}
    B_W[\mu_1,\mu_2]:=\frac{1}{2}\iint_{\real^N\times\R^N} W(x-y)\,d\mu_1(x)d\mu_2(y) .
\end{equation}
Obviously we have that $E_W[\mu]=B_W[\mu,\mu]$. Let us define the
shortcut notation
$T_W[\mu_1,\mu_2]:=E_W[\mu_1]-2B_W[\mu_1,\mu_2]+E_W[\mu_2]$ which
will often occur in several computations. For notational
simplicity, we will drop the subscript for $E_W$, $B_W$, and $T_W$
in detailed proofs while kept in the main statements.

Let us give a brief self-contained summary of the main concepts
related to distances between measures in optimal transport theory,
we refer to \cite{MR1964483,GS,MR2219334} for further details. A
probability measure $\pi$ on the product space $\real^N \times
\real^N$ is said to be a transference plan between $\mu \in
\prob(\real^N)$ and $\nu \in \prob(\real^N)$  if
\begin{equation}\label{marginal}
\pi(A \times \real^N)=\mu(A) \quad \text{and} \quad \pi(\real^N
\times A)=\nu(A)
\end{equation}
for all $A \in \B(\real^N)$. If $\mu, \nu \in \prob(\real^N)$,
then
$$
\Pi(\mu,\nu):=\{ \pi \in \prob(\real^N \times \real^N):
\eqref{marginal} \text{ holds for all }A \in \B(\real^N)\}
$$
denotes the set of admissible transference plans between $\mu$ and
$\nu$. Informally, if $\pi \in \Pi(\mu,\nu)$ then $d \pi(x,y)$
measures the amount of mass transferred from location $x$ to
location $y$. With this interpretation in mind note that
$\sup_{(x,y)\in \supp (\pi) } \abs{x-y}$ represents the maximum
distance that an infinitesimal element of mass from $\mu$ is moved
by the transference plan $\pi$. We will work with the
$\infty$-Wasserstein distance $d_\infty$ between two probability
measures $\mu,\,\nu$ defined by
\begin{equation}\label{disinfty}
    d_\infty(\mu,\nu) = \inf_{\pi \in \Pi(\mu,\nu)}  \sup_{(x,y)\in \supp (\pi) }
    \abs{x-y},
\end{equation}
which can take infinite values, but it is obviously finite for
compactly supported measures. This distance induces a complete
metric structure restricted to the set of probability measure with
finite moments of all orders, $\prob_\infty(\real^N)$, as proven
in \cite{GS}.

We remind that for $1\le p < \infty$ the distance $d_p$ between
two measures $\mu$ and $\nu$ is defined by
\[
    d_p^p(\mu,\nu)=\inf_{\pi\in\Pi(\mu,\nu)}\left\lbrace \iint_{\R^N\times \R^N}
    |x-y|^pd\pi(x,y)\right\rbrace.
\]
Note that $d_p(\mu,\nu)<\infty$ for $\mu,\nu\in \prob_p(\real^N)$
the set of probability measures with finite moments of order $p$.
Since $d_p(\mu,\nu)$ is increasing as a function of $1\le p <
\infty$, one can show that it converges to $d_\infty(\mu,\nu)$ as
$p\to\infty$. Since the distances are ordered with respect to $p$,
it is obvious that the topologies are also ordered. More
precisely, open sets for $d_p$ are always open sets for
$d_\infty$, and thus, $d_\infty$ induces the finest topology among
$d_p$, $1\leq p\leq \infty$. More properties of the distance
$d_\infty$ can be seen in \cite{MR2219334}.

Given $\mathcal T:\R^N \longrightarrow \R^N$ measurable, we say
that $\nu$ is the push-forward of $\mu$ through $\mathcal T$, $\nu
= \mathcal T\#\mu$, if $\nu [A] := \mu[\mathcal T^{-1}(A)]$ for
all measurable sets $A\subset\R^N$, equivalently
$$
\int_{\R^N} \varphi(x)\, d\nu(x) = \int_{\R^N} \varphi(\mathcal
T(x))\, d\mu(x)
$$
for all $\varphi\in C_b(\R^N)$. In case there is a map $\mathcal
T:\real^N\longrightarrow\real^N$ transporting $\mu$ onto $\nu$,
i.e. $\mathcal T \#\mu = \nu$, we immediately obtain
\[
    d_\infty(\mu,\nu) \leq \sup_{y\in\supp (\mu)} \abs{y-\mathcal T
    (y)}\,.
\]
This comes from \eqref{disinfty}, by using the transference plan
$\pi_{\mathcal T}= (\mathds{1}_{\real^N}\times \mathcal{T})\#\mu$.

\begin{lemma}\label{tech}
Assume that $\mu,\tilde \mu\in \prob(\real^N)$ are two convex
combinations: $\mu= m_0 \mu_0 + m_1 \mu_1$ and $\tilde{\mu} = m_0
\tilde \mu_0 + m_1 \mu_1$, where $\mu_0$ and $\tilde \mu_0$ are
supported in $B(x_0,\epsilon)$ for some $x_0\in\real^N$ and
$\epsilon >0$. Then $d_\infty(\mu, \tilde\mu) \le 2 \epsilon$.
\end{lemma}

\begin{proof}
Let $\pi_1 \in \Pi(\mu_1, \mu_1)$ be the transport plan induced by
the identity map, that is
$$
\iint_{\real^N \times \real^N} \phi(x,y) d \pi_1(x,y) =
\int_{\real^N} \phi(x,x) d \mu_1 (x)
$$
and let $\pi_0 \in \Pi(\mu_0, \tilde \mu_0)$ be any transport plan
between $\mu_0$ and $\tilde \mu_0$. Note that $\pi = m_0 \pi_0
+m_1 \pi_1 \subset \Pi(\mu,\tilde \mu)$ and $\supp (\pi) = \supp
(\pi_0) \cup \supp (\pi_1)$. Since $\pi_1$ is supported on the
diagonal we have $\sup_{(x,y)\in \supp (\pi_1) } \abs{x-y}=0$. On
the other hand, $\supp (\pi_0) \subset \supp (\mu_0) \times \supp
(\tilde \mu_0) \subset B(x_0,\epsilon) \times B(x_0,\epsilon)$ and
therefore $\sup_{(x,y)\in \supp (\pi_0) } \abs{x-y} \le 2
\epsilon$. We conclude that $\sup_{(x,y)\in \supp (\pi)} \abs{x-y}
\le 2 \epsilon$ which implies   $ \inf_{\pi \in \Pi(\nu,\rho)}
\sup_{(x,y)\in \supp (\pi) } \abs{x-y} \le 2 \epsilon$.
\end{proof}


\section{Lower bound on the Hausdorff dimension of the support}

In this section we consider potentials which are strongly
repulsive at the origin and we prove that if $\Delta W \sim
-1/|x|^\beta$ as $x\to 0$, $0<\beta<N$, then the Hausdorff
dimension of the support of local minimizers of the  interaction
energy is greater or equal to $\beta$. Actually our result is
slightly stronger: we prove that if $\mu$ is a local minimizer
then the support of any part of $\mu$ has Hausdorff dimension
greater or equal to $\beta$. Let us illustrate the importance of
controlling not only the dimension of $\mu$, but also the
dimension of the parts of $\mu$. Suppose for example that $\Delta
W \sim -1/|x|$ as $x\to 0$, then our result implies that any part
of $\mu$ has Hausdorff dimension greater or equal to $1$. As a
consequence  $\mu$ can not have an atomic part. If $\Delta W \sim
-1/|x|^{1.5}$ as $x\to 0$ then $\mu$ can not  have a part
concentrated on a curve and so on.

\subsection{Hypotheses and statement of the main result}
In this section, we will assume that the potential $W: \real^N \to
(-\infty,+\infty]$ satisfies the following hypotheses:
\begin{enumerate}
\item[(H1)] $W$ is bounded from below.

\item[(H2)] $W$ is lower semicontinuous (l.s.c.).

\item[(H3)] $W$ is uniformly locally integrable: there exists $M>0$
such that  $\int_{B(x,1)} W(y) dy \le M$ for all $x \in \real^N$.
\end{enumerate}
In order to state the main results of this section we will also need the following two definitions:
\begin{definition}[Generalized Laplacian]  Suppose  $W: \real^N \to (-\infty,+\infty]$ is  locally integrable. The approximate Laplacian of $W$ is defined by
\begin{align*}
-\Delta^\epsilon W(x):=  \frac{2(N+2)}{\epsilon^2}\left( W(x)- \dashint_{B(0,\epsilon)} W(x+y) dy  \right) ,
\end{align*}
where $\dashint_{B(x_0,r)} f(x) dx$ stands for the average of $f$
over the ball of radius $r$ centered at $x_0$, and the generalized
Laplacian of $W$ is defined by
\begin{align*}
-\Delta^0 W(x):= \liminf_{n \to \infty} \left\{-\Delta^{(1/n)} W(x) \right\}.
\end{align*}
\end{definition}
\begin{definition}[$\beta$-repulsive potential] \label{def:rep}
Suppose  $W: \real^N \to (-\infty,+\infty]$ is  locally
integrable. $W$ is said to be $\beta$-repulsive  at  the origin if
there exists $\epsilon>0$ and $C>0$ such that
\begin{align}
&-\Delta^0 W(x)  \ge \frac{C}{ |x|^{\beta}} \qquad \text{for all } 0<|x|<\epsilon \label{def:rep1}
 \\  &-\Delta^0 W(0)=+\infty. \label{def:rep2}
\end{align}
\end{definition}
By doing a Taylor expansion one can easily check that $\Delta^0
W(x)=\Delta W(x)$ wherever $W$ is twice differentiable. In
particular  if  $W$ is twice differentiable  away from the origin
as it is often the case for potentials of interest, then
\eqref{def:rep1} simply means that  $-\Delta W(x)  \ge
{C}/|x|^{\beta}$ for all  $0<|x|<\epsilon$. The terminology
``$\beta$-repulsive'' is justified by the fact that the rate at
which $\Delta^0 W(x)$  goes to $-\infty$ as $x$ approaches the
origin quantifies the repulsive strength of the potential at the
origin, therefore the greater $\beta$ is the more  repulsive  the
potential is around the origin. This is the rigorous mathematical
formulation of what we meant in \eqref{paris}. Additionally to
hypotheses (H1)--(H3), we will need the following technical
assumption on the potential $W$:
\begin{enumerate}
\item[(H4) \quad \;] There exists $C^*>0$ such that
\begin{equation*}  \label{bddbelow}
\Delta^\epsilon W(x)< C^* \quad \forall x\in\real^N \text{ and } \forall \epsilon \in (0,1).
\end{equation*}
\end{enumerate}
We are now ready to state the main theorems of this section:

\begin{theorem}\label{dimension}
Suppose $W$ satisfies {\rm (H1)--(H4)} and let $\mu$ be a
compactly supported local minimizer of the interaction energy with
respect to the topology induced by $d_\infty$. If $W$ is
$\beta$-repulsive at the origin, $0<\beta < N$, then the Hausdorff
dimension of the support of any part of $\mu$ is greater or equal
to $\beta$.
\end{theorem}

\begin{remark}\label{dimension-loc}
Observe that {\rm (H3)} and {\rm (H4)} are conditions which
restrict the growth of the potential and its derivatives at
$\infty$. For instance, a potential growing algebraically at
$\infty$ does not satisfy those assumptions. However, if we are
only interested in the dimensionality of the support for compactly
supported local minimizers, Theorem {\rm\ref{dimension}} holds
under weaker assumptions not restricting the growth of the
potential at $\infty$. Namely, {\rm (H3)} and {\rm (H4)} can be
substituted by {\rm (H3-loc)} and {\rm (H4-loc)}:
\begin{enumerate}
\item[{\rm (H3-loc)}] $W$ is locally integrable.

\item[{\rm (H4-loc)}] For every compact subset $K$ of $\real^N$
there exists $C_K^*>0$ such that
\begin{equation*}  \label{bddbelow2}
\Delta^\epsilon W(x)< C_K^* \quad \forall x\in K \text{ and }
\forall \epsilon \in (0,1),
\end{equation*}
\end{enumerate}
with obvious changes in the proof.
\end{remark}

\begin{remark}
In Theorem {\rm\ref{dimension}} (resp. Remark
{\rm\ref{dimension-loc}}) potential $W$ is assumed to be
$\beta$-repulsive at the origin and to satisfy  hypotheses {\rm
(H1)--(H4)} (resp. {\rm (H1)-(H2)-(H3-loc)-(H4-loc)}). Whereas
hypotheses {\rm (H1)--(H3)} (resp. {\rm (H1)-(H2)-(H3-loc)}) are
easily verified for a given potential, hypotheses {\rm (H4)} or
{\rm (H4-loc)} and the $\beta$-repulsivity are not as transparent.
To clarify the meaning of these more technical assumptions let us
consider the case where $W$ is smooth away from the origin and
satisfies
\begin{equation} \label{aba}
-\Delta W(x)  \ge \frac{C}{|x|^{\beta}} \qquad  \text{for all  $0<|x|<\epsilon$}
\end{equation}
for some $0<\beta<N$. Such a potential satisfies \eqref{def:rep1}
as pointed out in the comment after Definition {\rm\ref{def:rep}}.
Moreover most potentials of interest satisfying \eqref{aba}  will
also satisfy \eqref{def:rep2} and either {\rm (H4)} or {\rm
(H4-loc)}, but of course this need to be checked case by case. In
subsection {\rm 3.3} we consider some typical repulsive-attractive
potentials satisfying \eqref{aba} and we show that they satisfy
\eqref{def:rep2} and either {\rm (H4)} or {\rm (H4-loc)} depending
on their behavior at infinity.
\end{remark}

\subsection{Proof of Theorem \ref{dimension}}
First note that without loss of generality we can replace
hypothesis $(H1)$ by
\begin{enumerate}
\item[(H1')] $W$ is nonnegative
\end{enumerate}
since adding a constant to the potential $W$ does not affect the
local minimizers of $E_W$. The following lemma is classical:

\begin{lemma} \label{lsc-conv}
Suppose $W$ satisfies {\rm (H1')} and {\rm (H2)} and let $\mu \in
\prob(\real^N)$.  Then the function $V_\mu:\real^N \to
[0,+\infty]$ defined by
$$
V_{\mu}(x)= (W*\mu)(x)= \int_{\real^N}W(x-y) d\mu(y)
$$
is lower semicontinuous.
\end{lemma}

\begin{proof}  Suppose $x_n \to x$, then by Fatou's lemma we have
\begin{align*}
V_{\mu}(x)= \int_{\real^N} W(x-y) \; d\mu(y)
& \le\int_{\real^N} \liminf_n W(x_n-y)  \; d\mu(y) \\
 & \le \liminf_n \int_{\real^N}  W(x_n-y)   \; d\mu(y)
= \liminf_n V_{\mu}(x_n)
\end{align*}
as desired.
\end{proof}

Suppose now that  $W$ satisfies (H1')--(H4). Note that hypothesis
(H4) implies that $-\Delta^0 W \ge -C^*$ and as a consequence, for
any $\mu\in \prob(\real^N)$,  the function
$$
(-\Delta^0 W * \mu)(x)= \int_{\real^N}(-\Delta^0 W)(x-y) d\mu(y)=
\int_{\real^N}\left[(-\Delta^0 W)(x-y)+C^*\right] d\mu(y) -C^*
$$
is defined for all $x$ and
$
-\Delta^0 W * \mu: \real^N \to [-C^*, +\infty].
$

\begin{lemma} \label{min-principle}
Suppose that  $W$ satisfies {\rm (H1')--(H4)} and let $\mu \in
\prob(\real^N)$. If $x_0$ is a local min of $V_\mu=W*\mu$, in the
sense that there exists $\epsilon_0>0$ such that
\begin{equation} \label{locmin}
 V_\mu(x_0) \le V_\mu(x) \text{ for almost every }x \in B(x_{0},\epsilon_0),
\end{equation}
then $(\Delta^0 W * \mu)(x_0) \ge 0$.
\end{lemma}
\begin{proof}
Assume that $x_0$ satisfies \eqref{locmin}. We first show that
$V_{\mu}(x_0)<+\infty$. If it were not the case we would have that
$V_\mu=+\infty$ a.e. in $B(x_{0},\epsilon_0)$. But hypothesis (H3)
and Fubini's Theorem imply that
$$
\int_{B(x_0,1)} V_{\mu}(x) \,dx \le
\int_{\real^N}\int_{B(x_0-y,1)} W(z) \,dz \,d\mu(y)\leq M
$$
and therefore $V_{\mu}$ is finite almost everywhere in $B(x_0,1)$,
contradicting the fact that  $V_\mu=+\infty$ a.e. in
$B(x_{0},\epsilon_0)$. Now, for $\epsilon\le \epsilon_0$ we have
\begin{align}
0 &\le \frac{2(N+2)}{\epsilon^2}
\left(\dashint_{B(0,\epsilon)} V_{\mu}(x_0+x) dx  -V_{\mu}(x_0)\right) \nonumber\\
&=\frac{2(N+2)}{\epsilon^2}  \left(\int_{\real^N}
\dashint_{B(0,\epsilon)} W(x_0+x-y) dx d\mu(y)  -
\int_{\real^N}W(x_0-y) d\mu(y) \right) \,.\label{loulou}
\end{align}
Note that hypothesis (H4) implies that
\begin{equation*} 
 \dashint_{B(0,\epsilon)} W(x_0+x-y) dx \le W(x_0-y)+ \frac{C^* \epsilon^2}{2(d+2)}.
\end{equation*}
Since $V_{\mu}(x_0)<+\infty$, the functions $y\mapsto W(x_0-y)$
and $y\mapsto \dashint_{B(0,\epsilon)} W(x_0+x-y) dx$ are
$\mu$-integrable and the difference of the integrals in
\eqref{loulou} is equal to the integral of the difference.
Therefore we have:
\begin{equation} \label{loulou3}
0 \le   \int_{\real^N}
\frac{2(d+2)}{\epsilon^2}\left(\dashint_{B(0,\epsilon)} W(x_0-y+x)
dx   -  W(x_0-y) \right) d\mu(y) = \int_{\real^N} \Delta^\epsilon
W(x_0-y) d\mu(y)\,.
\end{equation}
Because of hypothesis (H4), we have that $-\Delta^\epsilon W+C^*
\ge 0$ for all $\epsilon \in (0,1)$. Therefore using Fatou's Lemma
and \eqref{loulou3}:
\begin{align*} 
 \int_{\real^N} \liminf_{n\to\infty} & \left\{ -\Delta^{(1/n)} W(x_0-y) + C^* \right\} d\mu(y)
 \nonumber\\
 &\le \liminf_{n\to\infty} \int_{\real^N} \left[-\Delta^{(1/n)} W(x_0-y) + C^*\right] d\mu(y) \le
 C^*\,,
\end{align*}
that is,  $(\Delta^{0} W * \mu)(x_0) \ge 0$.
\end{proof}

\begin{proposition} \label{constant}
Suppose that $W$ satisfies {\rm (H1')-(H2)-(H3)}. Let $\mu$ be a
local minimizer of the interaction energy with respect to the
$d_\infty$ and assume that $E[\mu]<+\infty$. Then any point $x_0
\in \text{supp }(\mu)$ is a local minimizer of $V_\mu$,
 in the sense that there exists $\epsilon_0>0$ such that
\begin{equation*}
V_\mu(x_0) \le V_\mu(x) \text{ for almost every }x \in B(x_0,\epsilon_0).
\end{equation*}
\end{proposition}
\begin{proof}
We argue by contradiction. Assume that there exists $x_0\in
\textrm{supp}(\mu)$ which is not a local minimum of $V_\mu$.  Fix
$\epsilon>0$. Then there exists a set $A\subset B(x_0,\epsilon)$
of positive Lebesgue measure, such that for $x\in A$,
$V_\mu(x)<V_\mu(x_0)$. The set $A$ can be written as follows:
$$
A=\cup_{n=1}^\infty\{ x\in A;\, V_{\mu}(x)\leq V_{\mu}(x_0)-1/n\},
$$
that is $A$ is an increasing union of measurable sets. Thanks to
the continuity from below of the Lebesgue measure, it implies that
$$
0<|A|=\lim_{n\to\infty}|\{ x\in A;\, V_{\mu}(x)\leq V_{\mu}(x_0)-1/n\}|,
$$
and there exists $n_0$ such that $\tilde A:=\{ x\in A;\,
V_{\mu}(x)\leq V_{\mu}(x_0)-1/{n_0}\}$ is of positive Lebesgue
measure. Thanks to the lower semicontinuity of $V_\mu$, there
exists  $\eta\in (0,\epsilon)$ such that
\begin{equation} \label{bibi}
\inf_{B(x_0,\eta)} V_\mu\geq V_\mu(x_0)-\frac 1{2n_0}\geq \sup_{\tilde A}V_\mu+\frac 1{2n_0}.
\end{equation}
Notice that $x_0\in \textrm{supp}(\mu)$ implies
$\mu(B(x_0,\eta))>0$. We can therefore define the probability
measures $\mu_0,\,\mu_{\tilde A}$ by
$$
 \mu_0(B)= \frac{1}{m_0} \mu(B \cap B(x_0,\eta)), \quad  \mu_{\tilde A}(B)= \frac{1}{|\tilde A|} |B\cap \tilde A|
$$
for any Borel set $B \in \B(\real^N)$, where
$m_0:=\mu(B(x_0,\eta))$. Let us now write $\mu$ as a convex
combination $\mu=m_0 \mu_0+ m_1 \mu_1$, and define the curve of
measures
$$
\mu_t= (m_0-t)\mu_0+t\mu_{\tilde A}+m_1\mu_1.
$$
It is clear by construction that $\mu_t\in\prob(\R^N)$ for
$t\in[0,m_0]$. Note that $\mu_t$ is obtained from $\mu$ by
transporting an amount $t$ of mass from the region $B(x_0, \eta)$
and by distributing it uniformly in the region $\tilde A$. Since
both $B(x_0, \eta)$ and $\tilde A$ are contained in $B(x_0,
\epsilon)$, the mass is transported by a distance smaller than $2
\epsilon$ and therefore we have $d_{\infty}(\mu, \mu_t)\le
2\epsilon$, see Lemma \ref{tech} for details. Inequality
\eqref{bibi} shows that the function $V_\mu$ is greater on the
region $B(x_0, \eta)$ than on the region $\tilde A$, therefore one
would expect that transporting mass from one region to the other
will decrease the interaction energy. Indeed we will  show that
$E[\mu_t]< E[\mu]$ for $t$ small enough. Since $\epsilon$ was
arbitrary, this will imply that we can always find a probability
measure arbitrarily close to $\mu$ (in the sense of the
$d_\infty$) with strictly smaller energy. This is a contradiction
concluding the proof.

We  are left to show that $E[\mu_t]< E[\mu]$ for $t$ small enough.
Since $0\leq E[\mu]<\infty$ and given by
$$
E[\mu]=m_0^2 E[\mu_0]+2 m_0m_1 B[\mu_0,\mu_1]+ m_1^2E[\mu_1]\,
$$
then the three terms $E[\mu_0], B[\mu_0,\mu_1]$ and $E[\mu_1]$ are
all positive and finite.  Note that $E[\mu_{\tilde A}]$ is also
finite: indeed, since $W$ is locally integrable by (H3), we have
\begin{equation*}
 E[\mu_{\tilde A}] =
\iint_{\real^N\times\real^N} W(x-y)\,d\mu_{\tilde
A}(x)\,d\mu_{\tilde A}(y) \leq\frac 1{|\tilde
A|^2}\iint_{B(x_0,\epsilon)\times B(x_0,\epsilon)} W(x-y)\,dx\,dy
<+\infty.
\end{equation*}
From \eqref{bibi} and the fact that $B[\mu,\mu_0] \le
\frac{1}{m_0}E[\mu]<+\infty$, we also have that
  \begin{equation}
  B[\mu,\mu_{\tilde A}] +  \frac{1}{2n_0} \le   B[\mu,\mu_0] < +\infty \label{bibibi}.
  \end{equation}
Using all these, we can show that all combinations of the bilinear
form $B[\cdot,\cdot]$ for the measures $\mu_0$, $\mu_1$, and
$\mu_{\tilde A}$ are finite:
\begin{align}
& E[\mu_0]< +\infty \quad , \quad E[\mu_1]< +\infty \quad , \quad E[\mu_{\tilde A}]< +\infty \quad , \quad B[\mu_1,\mu_0] < +\infty\,,\label{hi}\\
&  B[\mu_{\tilde A}, \mu_0] \le \frac{1}{m_0} B[\mu_{\tilde A},
\mu] < +\infty \quad , \quad B[\mu_{\tilde A}, \mu_1] \le
\frac{1}{m_1} B[\mu_{\tilde A}, \mu] < +\infty\,, \label{haha}
\end{align}
where we have used \eqref{bibibi} in order to obtain \eqref{haha}.
Note that in \eqref{haha} we have assumed  $m_1\neq 0$. If $m_1=0$
then $\mu_1$ can be chosen to be zero and therefore $B[\mu_{\tilde
A}, \mu_1]<+\infty$ trivially holds. Using
\eqref{hi}-\eqref{haha}, we are allowed to expand $E[ \mu_t]$ as:
\begin{align}
E[ \mu_t]=&\, E[(m_0-t)\mu_0 + m_1\mu_1+ t \mu_{\tilde A}] \nonumber\\
=&\, (m_0-t)^2 E[\mu_0]+ m_1^2 E[\mu_1]+t^2E[\mu_{\tilde A}] \nonumber\\
& +2(m_0-t)m_1 B[\mu_0,\mu_1]+2 (m_0-t)t B[\mu_0,\mu_{\tilde A}]+2 m_1tB[\mu_1,\mu_{\tilde A}]\nonumber\\
=&\, m_0^2 E[\mu_0]+2 m_0m_1 B[\mu_0,\mu_1]+ m_1^2E[\mu_1] \nonumber\\
& +2t \Big(m_0 B[m_0,\mu_{\tilde A}]+m_1 B[\mu_1,\mu_{\tilde A}]\Big) - 2t \Big( m_0 B[\mu_0,\mu_0] + m_1 B[\mu_0,\mu_1]\Big)\nonumber\\
& + t^2 E[\mu_0]+t^2 E[\mu_{\tilde A}] - 2 t^2 B[\mu_0,\mu_{\tilde A}]\nonumber\\
=&\, E[\mu]+  2t \Big( B[\mu_{\tilde A},\mu]-B[\mu_0,\mu] \Big)+
t^2 T[\mu_0,\mu_{\tilde A}] \,.\label{batman2}
\end{align}
Note that in the above computations we have only used the
bilinearity of $B[\cdot,\cdot]$ on the space of positive measures.
However, a formal computation  using the bilinearity of
$B[\cdot,\cdot]$ on the space of signed measures leads to the same
result in a much simpler way:
 \begin{equation*}
 E[ \mu_t]= E[\mu - t \mu_0+t\mu_{\tilde A}]
= E[\mu]+  2t \Big( B[\mu_{\tilde A},\mu]-B[\mu_0,\mu] \Big)+ t^2
T[\mu_0,\mu_{\tilde A}].
\end{equation*}
To conclude the proof note that because of \eqref{bibibi} the term
$B[\mu_{\tilde A},\mu]-B[\mu_0,\mu]$ appearing in \eqref{batman2}
is strictly negative and since the term $T[\mu_0,\mu_{\tilde
A}]=E[\mu_0]- 2 B[\mu_0,\mu_{\tilde A}]+ E[\mu_{\tilde A}] $ is
finite we can choose $t$ small enough  so  that  $E[\mu_t]<
E[\mu]$. This concludes the proof.
\end{proof}

Under the additional assumption that $W$ is not singular at the
origin, we can obtain a slightly stronger version of Proposition
\ref{constant} which will be needed in section 4.

\begin{proposition} \label{constant2}
Assume that $W$ and $\mu$ satisfy the same hypotheses than in
Proposition {\rm\ref{constant}}. Assume moreover that
$W(0)<+\infty$. Then any point $x_0 \in \supp(\mu)$ is a local
minimizer of $V_\mu$ in the classical sense and $V_\mu$  is
constant on any connected compact set  $K \subset \supp(\mu)$.
\end{proposition}
\begin{proof}
The proof of the first statement is similar to the proof of
Proposition \ref{constant}. We argue by contradiction: assume that
$\mu\in\prob(\R^N)$ is a local minimizer of $E[\cdot]$ and that
there exists $x_0 \in \supp (\mu)$ which is not a (classical)
local minimum of $V_\mu$. Fix $\epsilon>0$, then there exists
$x_a\in B(x_0,\epsilon)$ such that $V_{\mu}(x_a)< V_{\mu}(x_0)$.
But since $V_{\mu}$ is l.s.c. there exists $0<\eta<\epsilon$ such
that
  \begin{equation} \label{lsc}
  V_{\mu}(x_a)< V_{\mu}(x_0)\leq V_{\mu}(x) \quad \text{for all } x \in  B(x_0,\eta).
  \end{equation}
We then define $\mu_0$ and $\mu_1$ as in the proof of Proposition
\ref{constant}. The different idea now is to send mass from
$\mu_0$ to a Dirac Delta at the location $x_a$ instead of
distributing it evenly over a set $\tilde{A}$ of nonzero Lebesgue
measure: instead of letting $\mu_t= (m_0-t)\mu_0+t\mu_{\tilde
A}+m_1\mu_1$ as before, we now define $\mu_t=
(m_0-t)\mu_0+t\delta_{x_a}+m_1\mu_1$. The same expansion leads to
\begin{equation*}
 E[ \mu_t]= E[\mu - t \mu_0+t \delta_{x_a}]
= E[\mu]+  2t \Big( B[\delta_{x_a},\mu]-B[\mu_0,\mu] \Big)+ t^2
T[\mu_0,\delta_{x_a}]\,.
\end{equation*}
From \eqref{lsc} we obtain that the term
$B[\delta_{x_a},\mu]-B[\mu_0,\mu]$ is strictly negative. In order
to conclude the argument we need the term
$T[\mu_0,\delta_{x_a}]=E[\mu_0]- 2 B[\mu_0,\delta_{x_a}]+
E[\delta_{x_a}]$ to be finite. Note that $E[\delta_{x_a}]= W(0)/2$
therefore it is necessary for $W(0)$ to be finite in order to
conclude the proof.

We now prove the second statement. We follow classical arguments
from potential theory, see \cite[Proposition 0.4]{Papa} for
instance. Let $K$ be a connected compact set contained in
$\supp(\mu)$ and consider the sets $A=\{x\in K: V_{\mu}(x)>\inf_K
V_{\mu} \}$ and $B=\{x\in K: V_{\mu}(x)=\inf_K V_{\mu} \}$. Since
$V_{\mu}$ is l.s.c. the set $A$ is open relative to $K$. Let us
show that  $B$ is also open relative to $K$. We argue by
contradiction. Suppose there exists $x_a \in B$ such that for
every $\epsilon>0$ there exists $ x_{0,\epsilon}\in K$ with
$|x_a-x_{0,\epsilon}|< \epsilon$ and
$V_{\mu}(x_a)<V_{\mu}(x_{0,\epsilon})$.  Then following the exact
same steps as in the proof of the first statement we can construct
a probability measure with lower energy than $\mu$ and whose
$d_\infty$ distance to $\mu$ is smaller than $\epsilon$, therefore
leading to a contradiction and proving that $B$ is open relative
to $K$. Since $K$ is connected then either $A$ or $B$ must be
empty. But since $V_{\mu}$ is l.s.c it has to reach its minimum on
compact sets and therefore $A=\emptyset$ and $B=K$.
\end{proof}

\begin{remark}
Since $\supp(\mu)$ is closed,  the connected component of $\supp
(\mu)$ are also closed. So the second statement of Proposition
{\rm\ref{constant2}} implies that $V_{\mu}$ is constant on any
bounded connected component of $\supp (\mu)$.  In particular if
$\mu$ is compactly supported then  $V_{\mu}$ is constant on any
connected component of $\supp(\mu)$.
\end{remark}

Combining Lemma \ref{min-principle} and Proposition \ref{constant}
we obtain:
\begin{corollary}\label{posivityDelta}
Suppose that $W$ satisfies {\rm (H1')--(H4)}. If $\mu$ is local
minimizer of the interaction energy with respect to $d_\infty$ and
$E[\mu]<+\infty$, then $(\Delta^0 W*\mu)(x) \ge0$ for all  $x \in
\supp (\mu)$.
\end{corollary}
We recall the following result from \cite[Theorem 4.13]{falconer}.

\begin{proposition}\label{falconer}
Let $A$ be a Borel subset of $\mathbb R^N$, and $s\geq 0$. If
there exists a probability measure $\mu\in \mathcal P(\mathbb
R^N)$ supported on $A$ such that
\begin{equation*}
 \iint_{\real^N\times\real^N} \frac{d\mu(x)\,d\mu(y)}{|x-y|^s}<\infty,
\end{equation*}
then $\text{dim}_H A\geq s$, with $\text{dim}_H$ being the
Hausdorff dimension of $A$.
\end{proposition}

We are now ready to prove the main theorem.

\begin{proof}[Proof of Theorem \ref{dimension}]
Let $\rho$ be a nonzero part of $\mu$, that is $\mu= \rho+\nu$ for
some  nonnegative measure $\nu$. Let $A=\supp (\rho)$ and let us show
that $\text{dim}_H A\geq \beta$. Choose $\epsilon$ small enough so
that \eqref{def:rep1} holds, choose $x_0\in A$ and define the
measure
$$
 \mu_0(B)= \rho(B \cap B(x_0, \epsilon/2)).
$$
Clearly $\mu$ can be written $\mu= \mu_0+\mu_1 $, where $\mu_0$
and $\mu_1$ are two (nonnegative) measures of mass $m_0>0$ and
$m_1\ge 0$ and where $\mu_0$ is supported in $A\cap B(x_0,
\epsilon/2)$. Then from \eqref{def:rep1} we get:
 \begin{align*}
{C} \iint_{\real^N\times\real^N} & \frac{d\mu_0(x)\,d\mu_0(y)}{|x-y|^{\beta}}  \le   \iint_{\real^N\times\real^N} - \Delta^0 W(x-y) d \mu_0(x) d \mu_0(y) \nonumber\\
& =   \iint_{\real^N\times\real^N} - \Delta^0 W(x-y) d \mu(x) d \mu_0(y)  -    \iint_{\real^N\times\real^N} -\Delta^0 W(x-y) d \mu_1(x) d \mu_0(y) \nonumber\\
      & =  -   \int_{\real^N\times\real^N} ( \Delta^0 W*\mu)(y) d \mu_0(y)   +    \iint_{\real^N\times\real^N}  \Delta^0 W(x-y) d \mu_1(x) d \mu_0(y)\\
       & \le      \iint_{\real^N\times\real^N} \Delta^0 W(x-y) d \mu_1(x) d \mu_0(y) \le    C^* m_1 m_0 <+ \infty.
 \end{align*}
We have used the fact that $\Delta^0 W*\mu$ is nonnegative on the
support of $\mu$ from Corollary \ref{posivityDelta} and that
$\Delta^0 W(x)< C^*$ by hypothesis (H4). We then apply Proposition
\ref{falconer} to the probability measure $\mu_0/m_0$, which is
supported on $A$, to obtain $\text{dim}_H A\geq \beta$.
\end{proof}


\subsection{Example of potentials satisfying the hypotheses of Theorem \ref{dimension}}
In this subsection we  consider the class of potentials
$$
W_\alpha(x)=c \, h_\alpha(x)+ \psi(x)
$$
where $\psi \in C^3(\real^N)$ bounded from below, $c>0$ and
$h_\alpha: \real^N \to (-\infty, \infty]$ is the power-law
function:
$$
h_\alpha(x)= - |x|^\alpha /\alpha
$$
for $x\neq 0$ and $\alpha\in\real$ with the convention
$h_0(x)=-\log{|x|}$. We define $h_\alpha(0) = 0$ if $\alpha > 0$
or $h_\alpha(0) =+\infty$ if $\alpha \leq 0$. The potentials
$W_\alpha$ are typical examples of repulsive-attractive potentials
behaving like $-|x|^\alpha/\alpha$ around the origin. It is
trivial to check that $W_\alpha$ satisfies (H1)-(H2)-(H3-loc) for
any $\alpha>-N$ (in the case $\alpha\geq 0$ the function $\psi$
need to grow fast enough at infinity for hypothesis (H1) to hold).
Note also that for $x \neq 0$ we have
\begin{equation} \label{beta-rep}
-\Delta W_\alpha(x)=c\frac{(\alpha+N-2)}{ |x|^{2-\alpha}}  - \Delta \psi(x)
\end{equation}
and therefore  if $\alpha+N-2 > 0$ then $W_\alpha$ satisfies
\eqref{def:rep1} from the definition of $\beta$-repulsivity with
$\beta=2-\alpha$. The goal of this subsection is to show that
$W_\alpha$ also satisfies \eqref{def:rep2} and (H4-loc).

We start by checking \eqref{def:rep2}. An explicit computation gives
\begin{equation*}
-\Delta^\epsilon h_\alpha(0)=   \frac{2(N+2)}{ \epsilon^2} \left( h_\alpha(0)- \dashint_{B(0,\epsilon)}  h_\alpha(y) dy  \right) =
\begin{cases}  2(N+2)  \frac{N}{N+\alpha} \frac{\epsilon^{\alpha-2}}{\alpha}  & \text{ if } \alpha>0\\
+\infty & \text{ if } 2-N\le\alpha \le 0
\end{cases}
\end{equation*}
where we have used the fact that  $h_\alpha(0)=0$ for $\alpha>0$
and $h_\alpha(0)=+\infty$  for $\alpha\le 0$. Letting $\epsilon
\to 0$ and using the fact that $\Delta \psi(0)$ is finite we
obtain
\begin{equation} \label{beta-rep2}
-\Delta^0 W(0)=+\infty \quad \text{for all } \alpha<2.
\end{equation}
Combining \eqref{beta-rep} and \eqref{beta-rep2} we see that for
$2-N < \alpha<2$ the potential $W_\alpha$ is  $\beta$-repulsive
with $\beta=2-\alpha \in (0,N)$.

We now show that for $\alpha > 2-N$ the potentials $W_\alpha$
satisfies hypothesis (H4-loc). The key point is that the functions
$h_\alpha$ are superharmonic for $ \alpha > 2-N $.  Let us recall
the definition of superharmonicity:
\begin{definition}
A lower semicontinuous function $h: \real^N \to (-\infty, \infty]$
is said to be superharmonic on  the connected open set $\Omega$ if
it is not identically equal to $+\infty$ on $\Omega$ and if
$$
h(x)\ge \dashint_{B(x,r)} h(y) dy
$$
for all $x\in \Omega$ and $r>0$ such that $B(x,r) \subset \Omega$.
\end{definition}

We also recall that if $h \in C^2(\Omega)$, then $h$ is
superharmonic on $\Omega$ if and only if $\Delta h(x) \le 0$ for
all $x \in \Omega$. To see that the functions $h_\alpha$ are
superharmonic for  $ \alpha > 2-N $, first note that for $x \neq
0$
$$
\Delta h_\alpha (x)=-\frac{(\alpha+N-2)}{ |x|^{2-\alpha}} \le 0.
$$
Therefore $h_\alpha$ is superharmonic on $\real^N \backslash
\{0\}$ and it can be easily checked that it satisfies the
super-mean value property at the origin \cite[Definition
2.1]{Papa}. Both together imply that it is actually superharmonic
on the full space $\real^N$, \cite[Corollary 2.1]{Papa}. As a
consequence we directly obtain from the definition of the
approximate Laplacian that $\Delta^\epsilon h_\alpha(x) \le 0$ and
therefore
$$
\Delta^\epsilon W_\alpha = c\Delta^\epsilon h_\alpha + \Delta^\epsilon \psi \le  \Delta^\epsilon \psi.
$$
To conclude we note that since $\psi \in C^3(\real^N)$ a simple
Taylor expansion shows that $\Delta^\epsilon \psi$ converges
uniformly to $\Delta \psi$ on compact sets. Indeed, the expansion
gives
\begin{align}
\Delta^\epsilon \psi(x) &= \frac{2(N+2)}{\epsilon^2}\dashint_{B(0,\epsilon)} \psi(x+y)-\psi(x) \;  dy\nonumber\\
&= \frac{2(N+2)}{\epsilon^2} \dashint_{B(0,\epsilon)} y^T\nabla\psi(x)+ y^T H\psi(x) y + O(\epsilon^3) \;  dy \label{bouba1}\\
&= \frac{2(N+2)}{\epsilon^2}\left(  \Delta \psi(x) \dashint_{B(0,\epsilon)}  y_1^2 dy + O(\epsilon^3) \right) \label{bouba2}\\
&=  \Delta \psi(x) + O(\epsilon)\nonumber
\end{align}
To go from \eqref{bouba1} to \eqref{bouba2} we have used the fact
that most of the terms in the Taylor expansion are equal to zero
after integrating on spheres due to symmetry. The only remaining
terms are the ones involved in the Laplacian. Note that since the
partial derivative of order 3 of $\psi$ are bounded on compact
subsets of $\real^N$, then the error term is uniform for $x$ in
compact sets. Since $\Delta \psi$ is bounded on compact sets, we
conclude that for $\alpha > 2-N$ the potential $W_\alpha$
satisfies (H4-loc).

Moreover, let us point out that if $\psi$ is well behaved at
infinity in terms of regularity and growth, then $W_\alpha$
satisfies either (H1)--(H4) or (H1)-(H2)-(H3-loc)-(H4-loc). We
summarize this discussion in the following proposition:

\begin{proposition}
If  $2-N < \alpha <2 $ and if $\psi \in C^3(\real^N)$ then
$W_\alpha$ is $(2-\alpha)$-repulsive around the origin and
satisfies {\rm (H4-loc)}.
\end{proposition}


\section{Mild Repulsion implies 0-dimensionality}

In this section, we will show that if the potential is mildly
repulsive, meaning that it behaves locally near zero like
$-|x|^\alpha$ with $\alpha>2$, then the support of the measure
cannot contain measures concentrated on smooth manifolds of any
dimension except 0-dimensional sets.

\begin{definition}\label{kdimensional}
Let $1 \le k \le N$. A probability measure $\mu \in
\prob(\real^N)$ is said to have a regular $k$-dimensional
part if it can be written
$$
\mu= \mu_1+\mu_2
$$
where $\mu_1$ is a nonnegative measure on $\real^N$ and $\mu_2$ is
defined by
$$
\int_{\real^N} \psi(x) d\mu_2(x)= \int_{\mathcal M} \psi(x)f(x) d
\sigma(x)  \quad \forall \psi \in C_0(\real^N)
$$
for some $C^2$ manifold $\mathcal M$ of dimension $k$ and a non
identically zero nonnegative function $f: \mathcal M \to
(0,+\infty]$ integrable with respect to the surface measure
$d\sigma(x)$ on $\mathcal M$. Moreover to avoid pathological
situations, we assume that there exists $x_0 \in \M$, $c,\kappa>0$
such that
\begin{equation}\label{positivity}
 f(x)\geq c \qquad \forall x\in {\mathcal M} \cap B(x_0,\kappa).
\end{equation}
\end{definition}

We now state the main result of this section:
\begin{theorem} \label{thm1}
Let $W \in C^2(\real^N)$ be a radially symmetric potential which
is equal to $-|x|^\alpha/\alpha$ in a neighborhood of the origin.
If $\alpha>2$ then a local minimizer of the interaction energy
with respect to $d_\infty$ cannot have a $k$-dimensional component
for any $1\le k\le N$.
\end{theorem}

For the above theorem to be true it is not necessary for the
potential to be exactly equal to a power law $-|x|^\alpha$,
$\alpha>2$, around the origin. It is enough for the potential to
behaves like $-|x|^\alpha$, $\alpha>2$, at the origin in a precise
convexity sense -- see Theorem \ref{general-thm}.

\subsection{Preliminaries on convexity}

To prove Theorem \ref{thm1}, we need
some convex analysis concepts, see
\cite{Carrillo-McCann-Villani06,Ambrosio2008} and the references
therein. The term {\em modulus of convexity} refers to any convex
function $\phi$ on the positive reals satisfying
\begin{eqnarray*}
(\phi_0)&& \phi:[0,\infty) \longrightarrow \real\
        \mbox{is continuous and vanishes only at $\phi(0)=0$.} \\
(\phi_1)&& {\phi(x) \geq -kx {\rm\ for\ some\ } k<\infty.}
                                           \label{suplinear_cheat}
\end{eqnarray*}
Now, we can quantify the convexity of certain functions in terms
of a modulus of convexity.

\begin{definition} A function $h:[0,+\infty) \to \real$ is $\phi-$uniformly convex on
$(a,b)$ if there exists a modulus of convexity $\phi$ such that
 \begin{equation} \label{phi-convexity}
 h\left(\frac{r_1+r_2}{2}\right)\le \frac 12\left( h(r_1) +h(r_2)\right) - \frac14\int_0^{|r_1-r_2|}\phi(t)\,dt,
 \end{equation}
for all $r_1,r_2 \in (a,b)$.

A function $h:[0,+\infty) \to \real$ is $\lambda-$convex on
$(a,b)$ if it is $\phi$-uniformly convex with $\phi(s)=\lambda s$
and $\lambda\in\real$.
\end{definition}

Note that if $h$ is $\lambda-$convex, then \eqref{phi-convexity}
reads
\begin{equation} \label{lambda-convexity}
 h\left(\frac{r_1+r_2}{2}\right)\le \frac12 h(r_1) + \frac12 h(r_2) - \frac{\lambda}{8} (r_1-r_2)^2,
\end{equation}
for all $r_1,r_2 \in (a,b)$. It is equivalent to assume that the
function $h(r)-\tfrac{\lambda}{2} r^2$ is convex on $(a,b)$. The
following proposition can be easily proven:
\begin{proposition}[Convexity properties of power laws] \label{robin}

\

\begin{enumerate}
 \item[(i)] If $q\in (1,2]$, then
$h(r)={r^q}$ is $\lambda-$convex on $[0,R]$ for
$\lambda=\inf_{(0,R)} h''=q(q-1)R^{q-2}>0$, and thus, uniformly
convex on $[0,R]$.

\item[(ii)] If $q> 2$, then $h(r)={r^q}$ is $\phi-$uniformly
convex on $\mathbb R_+$, with $\phi(t)=2^{2-q}t^{q-1}/q$.  That is
\begin{equation}\label{q-convexity}
 h\left(\frac{r_1+r_2}{2}\right) \le \frac 12\left( h(r_1) +h(r_2)\right) -2^{-q}|r_1-r_2|^q,
\end{equation}
for $r_1,\,r_2\geq 0$.
\end{enumerate}
\end{proposition}

\subsection{Proof of Theorem \ref{thm1}}
In this subsection we prove the following generalization of
Theorem~\ref{thm1}.

\begin{theorem} \label{general-thm}
Let $W(x)= w(|x|)$ be continuously differentiable, bounded from
below, and decreasing as a function of $|x|$ in a neighborhood of
the origin. Assume moreover that $W$ behaves like the power law
$-|x|^\alpha$, $\alpha>2$, near the origin, in the sense that for
some $C^\ast>0$ and $R>0$ small enough,  $W(r)=-h(r^2)$ satisfies:
\begin{itemize}
\item if $\alpha\in (2,4]$, $h$ is $\lambda-$convex on $[0,R]$
with $\lambda=C^\ast R^{\alpha/2-2}$.

\item if $\alpha\in (4,\infty)$, $h$ is $\phi-$uniformly convex on
$[0,R]$, with $\phi(t)=C^\ast t^{\alpha/2-1}$,
\end{itemize}
and $C^\ast|w'(r)|\leq r^{\alpha-1}$ on $[0,R]$. Then a local
minimizer of the interaction energy with respect to $d_\infty$
cannot have a $k$-dimensional part for any $1\le k\le N$.
\end{theorem}

Theorem \ref{thm1} is a direct consequence of Theorem
\ref{general-thm}, thanks to Proposition \ref{robin}.

\

We first provide an explicit formula for how the energy
changes when perturbing a local minimizer:
\begin{lemma} \label{localization}
Suppose that $W: \real^N \to (-\infty,+\infty]$ is symmetric,
l.s.c. and bounded from below with $W(0)<+\infty$. Let
$\mu\in\prob(\R^N)$ be a local minimizer of the interaction energy
with respect to $d_\infty$ and $E[\mu]<+\infty$. Given a connected
domain $\Omega\subseteq\supp(\mu)$, a Borel map
$\pi:\Omega\to\Omega$ and a convex decomposition
$\mu=m_1\mu_1+m_2\mu_2$ with $\supp(\mu_1)\subset \Omega$, we
deduce that
\begin{align}
E[m_1(\pi\#\mu_1)+m_2\mu_2]-E[m_1\mu_1+m_2\mu_2]=m_1^2
\,T[\pi\#\mu_1,\mu_1]\,. \label{main}
\end{align}
\end{lemma}

\begin{proof} Since $B$ in \eqref{defB} is a bilinear form,
\begin{align}
E[m_1(\pi\#\mu_1)+m_2\mu_2]&-E[m_1\mu_1+m_2\mu_2] \nonumber\\
&\quad=m_1^2B[\pi\#\mu_1,\pi\#\mu_1]+2m_1m_2B[\pi\#\mu_1,\mu_2] \nonumber\\
&\qquad- m_1^2B[\mu_1,\mu_1]-2m_1m_2B[\mu_1,\mu_2].\label{tech1}
\end{align}
We now use the fact that $\mu$ is local minimizer to express the
terms involving $\mu_2$  as terms involving only $\mu_1$.
Proposition \ref{constant2} implies that the function $V_\mu(x)$
is constant on the connected domain $\Omega$ and since
$\pi(\Omega) \subset \Omega$ we have:
\[
    \int_{\R^N} W(\pi(x)-y)\,d\mu(y)=\int_{\R^N} W(x-y)\,d\mu(y)\quad \forall x\in\Omega.
\]
and therefore, since $\mu=m_1\mu_1+m_2\mu_2$,
\begin{multline*}
m_1\int_{\R^N} W(\pi(x)-y)\,d\mu_1(y)+m_2\int_{\R^N} W(\pi(x)-y)\,d\mu_2(y)\\
=m_1\int_{\R^N} W(x-y)\,d\mu_1(y)+m_2\int_{\R^N} W(x-y)\,d\mu_2(y)
\end{multline*}
for all $x\in \Omega$.  Since $\supp(\mu_1)\subset \Omega$ we can
integrate both sides with respect to $d\mu_1(x)$ and obtain, after
multiplication by $m_1$:
\begin{multline*}
m_1^2\int_{\R^N\times\R^N}\!\!\!\!\!\!\!\!\!\!\!\! W(\pi(x)-y)\,d\mu_1(y)d\mu_1(x)+m_1m_2\int_{\R^N\times\R^N}\!\!\!\!\!\!\!\!\!\!\!\! W(\pi(x)-y)\,d\mu_2(y)\,d\mu_1(x)\\
=m_1^2\int_{\R^N\times\R^N}\!\!\!\!\!\!\!\!\!\!\!\!
W(x-y)\,d\mu_1(y)d\mu_1(x)+m_1m_2\int_{\R^N\times\R^N}\!\!\!\!\!\!\!\!\!\!\!\!
W(x-y)\,d\mu_2(y)\,d\mu_1(x)
\end{multline*}
or equivalently, using the $B$-notation,
\begin{equation*}
2 m_1^2 B[\mu_1,\pi\#\mu_1]+ 2m_1m_2 B[\mu_2,\pi\#\mu_1] = 2m_1^2
B[\mu_1,\mu_1] + 2m_1 m_2 B[\mu_2,\mu_1]
 \end{equation*}
and therefore rearranging the terms, we can express the terms
involving $\mu_2$ in terms of the ones involving only $\mu_1$:
\begin{equation*}
2m_1m_2 \Big[ B[\pi\#\mu_1, \mu_2]-B[\mu_1,\mu_2] \Big] = 2 m_1^2
\Big[B[\mu_1,\mu_1]- B[\mu_1,\pi\#\mu_1]  \Big].
 \end{equation*}
The desired identity \eqref{main} is readily obtained by plugging
the last equality into \eqref{tech1} reminding the definition of
$T[\mu,\nu]$ in Section 2.
\end{proof}

\begin{definition}
Let $1 \le k \le N$. We denote by $D^k_\varepsilon$ the
the $k$-dimensional disk of radius
$\varepsilon$:
$$
D^k_\varepsilon=\{ (x_1,\ldots,x_N): \;  x_1^2+\ldots+x_k^2 \le
\varepsilon^2\,\mbox{and } x_{k+1}=\dots=x_N=0\},
$$
and by $\nu_{\varepsilon,k} \in \prob(\real^N)$  the uniform
probability distribution on $D^k_\varepsilon$:
$$
\nu_{\varepsilon,k}=\frac1{|D^k_\varepsilon|}
\delta_{D^k_\varepsilon},
$$
where $|D_\varepsilon^k|$  is the Lebesgue
measure of dimension $k$ of $D_\varepsilon^k$, that is   $|D_\varepsilon^k|=\sigma_k \varepsilon^k$ where $\sigma_k$ the
area of the unit $k$-dimensional ball.
$\nu_{\varepsilon,k}$ then satisfies
$$
\int_{\R^N} \psi(x) d\nu_{\varepsilon,k}(x) = \frac
1{|D^k_\varepsilon|} \int_{x_1^2+\ldots+x_k^2 \le \varepsilon^2}
\psi(x_1,\dots,x_k,0,\dots,0) \,dx_1\dots dx_k
$$
for all $\psi\in C^0(\R^N)$.
 \end{definition}

The following Lemma combined with Lemma \ref{localization} shows
that if a flat $k$-dimensional disk is contained in the support of
a local minimizer, then the energy can be reduced by concentrating
all the mass contained in the disk into a single point. As a
consequence the support of a local minimizer cannot contain a
flat $k$-dimensional disk.

\begin{lemma} \label{convexity}
Suppose that $W(x)=-h(|x|^2)$ satisfies the assumptions of Theorem
{\rm\ref{general-thm}}. Then, there exists $c_{k,\alpha}>0$ such
that for $\varepsilon$ small enough,
\begin{equation*}
T[\delta_0,\nu_{\varepsilon,k}]=B[\delta_0,\delta_0] -2
B[\delta_0, \nu_{\varepsilon,k}]+ B[\nu_{\varepsilon,k} ,
\nu_{\varepsilon,k}] \le - c_{k,\alpha} \varepsilon^\alpha.
\end{equation*}
\end{lemma}

\begin{proof}
Since $W$ is bounded from below and $W(0)<+\infty$, we can assume
without loss of generality that $W(0)=0$ by adding to $W$ a
suitable constant. Then, the first term $B[\delta_0,\delta_0]$ is
equal to zero. Symmetrizing the integral involved in second term
we obtain:
\begin{align*}
B[\delta_0, \nu_{\varepsilon,k}]=-\frac{1}{2}\int_{\R^N} h(|y|^2)
d\nu_{\varepsilon,k}(y)&=-\frac{1}{2}\int_{\R^N\times\R^N}\!\!\!\!\!\!
h(|y|^2)\,d\nu_{\varepsilon,k}(x) d\nu_{\varepsilon,k}(y) \\ &=
-\frac{1}{4}\int_{\R^N\times\R^N}\!\!\!\!\!\!
\left[h(|x|^2)+h(|y|^2)\right]\,d\nu_{\varepsilon,k}(x)
d\nu_{\varepsilon,k}(y).
\end{align*}
Since $\nu_{\varepsilon,k}(-y)=\nu_{\varepsilon,k}(y)$, we can
also symmetrize the third term and obtain:
\begin{align*}
B[\nu_{\varepsilon,k},\nu_{\varepsilon,k}]&=
-\frac{1}{2}\int_{\R^N\times\R^N}\!\!\!\!\!\!
h(|x-y|^2)\,d\nu_{\varepsilon,k}(x) d\nu_{\varepsilon,k}(y) \\
&=- \frac{1}{4}\int_{\R^N\times\R^N}\!\!
\left[h(|x-y|^2)+h(|x+y|^2)\right]\;\; d\nu_{\varepsilon,k}(x)
d\nu_{\varepsilon,k}(y).
\end{align*}
Combining the three terms we find
\begin{equation} \label{baba}
T[\delta_0,\nu_{\varepsilon,k}] = \frac{1}{2}
\int_{\R^N\times\R^N} A(x,y)
 \;\; d\nu_{\varepsilon,k}(x) d\nu_{\varepsilon,k}(y)
\end{equation}
with
$$
A(x,y):=h(|x|^2)+h(|y|^2) - \frac{h(|x-y|^2)+h(|x+y|^2)}{2}\,.
$$
Under the assumptions of Theorem \ref{thm1}, $h$ is convex on
$(0,2\varepsilon^2)$ and since $h(0)=0$, we deduce
$$
h(r_i^2)\leq \frac{r_i^2}{r_1^2+r_2^2}h(r_1^2+r_2^2),
$$
for $r_i\geq 0$, $i=1,\,2$. Using the above inequalities for
$i=1,2$ we get
\begin{eqnarray*}
h(|x|^2)+h(|y|^2) \le h(|x|^2+|y|^2) = h\left(\frac{1}{2} |x+y|^2+
\frac{1}{2}|x-y|^2\right).
\end{eqnarray*}

In the rest of this Lemma, $C$ will denote some generic constant
that will change from step to step. For $\alpha\in(2,4]$, $h$ is
$\lambda-$convex on $(0,2\varepsilon^2)$ with
$\lambda=C\varepsilon^{\alpha-4}$, so that plugging into
\eqref{lambda-convexity}, we obtain
\begin{align}
h(|x|^2)+h(|y|^2) & \leq \frac{1}{2} h\left( |x+y|^2 \right)+
\frac{1}{2} h \left(|x-y|^2\right)-  C\varepsilon^{\alpha-4}
(x\cdot y)^2. \label{bibibi1}
\end{align}
Combining \eqref{baba} and \eqref{bibibi1} we get:
\begin{align*}
T[\delta_0,\nu_{\varepsilon,k}] & \le - C\varepsilon^{\alpha-4}
  \int_{\R^N\times\R^N} (x\cdot y)^2    \;\;  d\nu_{\varepsilon,k}(x) d\nu_{\varepsilon,k}(y) \\
  & =- C\varepsilon^\alpha
  \int_{\R^N\times\R^N} (x\cdot y)^2    \;\;  d\nu_{1,k}(x)
  d\nu_{1,k}(y)\,.
\end{align*}

For $\alpha\geq 4$, $h$ is $\phi-$convex with
$\phi(t)=Ct^{\alpha/2-1}$, so that plugging into
\eqref{q-convexity}, we obtain
\begin{align}
h(|x|^2)+h(|y|^2)\leq \frac{1}{2} h\left( |x+y|^2 \right)+
\frac{1}{2} h \left(|x-y|^2\right)-  C|x\cdot y|^{\alpha/2}.
\label{bibibi2}
\end{align}
Combining \eqref{baba} and \eqref{bibibi2} we get:
\begin{align*}
T[\delta_0,\nu_{\varepsilon,k}] &\le - C \int_{\R^N\times\R^N}
|x\cdot y|^{\alpha/2} \;\;
d\nu_{\varepsilon,k}(x) d\nu_{\varepsilon,k}(y) \\
&=- C\varepsilon^\alpha \int_{\R^N\times\R^N} |x\cdot
y|^{\alpha/2}    \;\; d\nu_{1,k}(x) d\nu_{1,k}(y) \,.
\end{align*}
\end{proof}

The last Lemma combined to Lemma \ref{localization} shows that the
support of a local minimizer cannot contain a flat $k$-dimensional
disk of radius $\epsilon$. To conclude the  proof of Theorem
\ref{general-thm}, we need to introduce some differential geometry
tools. Let $R>0$, and $g:D_R^k\to \real^{N-k}$ a $C^2$-function
such that $g(0)=0$, $\nabla g(0)=0$. We define the
parameterisation $P_g$ of the graph of $g$ as follows:
\begin{eqnarray}
P_g: D_R^k &\longrightarrow & \phantom{ds}\real^N,\label{Pg}\\
(x',0)\phantom{dsf}&\mapsto&(x',g(x'))\nonumber
\end{eqnarray}
where $x'=(x_1,\dots,x_k)$ stands for the $k$ first coordinates.
Let us remark that classical differential geometry implies that
any $C^2$-manifold can be locally parameterized by such graphs by
choosing conveniently the axis and reordering of variables.
Moreover, this can be done in such a way that the volume element
of the graph $J_g$ is as close to the unit volume element of the
flat tangent space by taking $R$ small enough, see
\cite{PRminimal}. More precisely, there exists a constant $C_g$
depending only on the second derivatives of $g$ on $D_R^k$ such
that
\begin{equation}\label{geodiff}
\|J_g-1\|_{L^\infty(D_\varepsilon^k)}\leq C_g\varepsilon,
\end{equation}
for $0<\varepsilon<R$ small enough.

\begin{lemma} \label{errorestimate}
If $W$ satisfies the assumptions of Theorem
{\rm\ref{general-thm}}, and $g\in C^2(D_R^k, \real^{N-k})$
satisfies $g(0)=0$, $\nabla g(0)=0$, then for $\varepsilon>0$
small enough,
\begin{align*}
T[\delta_0, P_g
\#\nu_{\varepsilon,k}]-T[\delta_0,\nu_{\varepsilon,k}] \le
\frac{2^{\alpha-1}\varepsilon^\alpha}{C^\ast}\|\nabla g\|_{L^\infty(D^k_\varepsilon)}.
\end{align*}
\end{lemma}

\begin{proof}
Note that by continuity for $\varepsilon>0$ small enough, we have
$\|\nabla g\|_{L^\infty(D^k_\varepsilon)}\leq 1$. We first point
out that
$$
T[\delta_0, P_g
\#\nu_{\varepsilon,k}]-T[\delta_0,\nu_{\varepsilon,k}]=\int_{\R^N\times\R^N}
A(x,y) d \nu_{\varepsilon,k}(x) d\nu_{\varepsilon,k}(y)
$$
with
$$
A(x,y)= \frac{w (|P_g(x)-P_g(y)|)-w(|x-y|)}{2}- \Big[
w(|P_g(x)|)-w(|x|)\Big]\,.
$$
Thanks to the definition of the parameterisation $P_g$,
$|P_g(x)-P_g(y)|\geq |x-y|$. Moreover since  $w$ is decreasing in
a neighborhood of the $0$, the first term in $A(x,y)$ is negative
for $\max(|x|,|y|)<\varepsilon$ small enough. To estimate the
second term, we use the mean value theorem for $g$ around $x'=0$,
remembering that $ g(0)=0$:
$$
|P_g(x)-(x',0)|=|g(x')-g(0)|\leq \varepsilon \|\nabla
g\|_{L^\infty(D^k_\varepsilon)},
$$
since $C^\ast|w'(r)|\leq r^{\alpha-1}$, we conclude
$$
w(|P_g(x)|)-w(|x|) \leq
\|w'\|_{L^\infty([0,2\varepsilon])}\varepsilon\|\nabla
g\|_{L^\infty(D^k_\varepsilon)} \leq
\frac{2^{\alpha-1}\varepsilon^\alpha}{C^\ast}\|\nabla
g\|_{L^\infty(D^k_\varepsilon)}\,.
$$
\end{proof}

\begin{proof}[Proof of the Theorem \ref{thm1}]
Assume that $\mu$ is a local minimizer of $E$ in $d_\infty$ and
that it has a regular $k$-dimensional part in the sense of
Definition \ref{kdimensional}. Let $\mathcal{M}$ be the
$C^1$-submanifold on which this component is supported, and $f$ be
the density on $\mathcal{M}$ of this component. Due to assumption
\eqref{positivity}  there exists $x_0 \in \M$, $c,\kappa>0$ such
that
$$
 f(x)\geq c \,,\qquad \forall x\in {\mathcal M} \cap B(x_0,\kappa).
$$
As discussed above and without loss of generality, we can assume
that $x_0$=0 and that $\mathcal M$ is locally the graph of a
$C^2$-function $g:D_R^k\to \real^{N-k}$, for some $\kappa>R>0$,
such that $g(0)=0$, $\nabla g(0)=0$.

Let $P_g$ be the parameterisation defined in \eqref{Pg}. Note that
for $\varepsilon\leq R$, $\mu_1^\varepsilon:=P_g\#
\nu_{\varepsilon,k}\in\prob(\R^N)$ is absolutely continuous with
respect to the volume element on $\mathcal{M}$ with a density
denoted still by $\mu_1^\varepsilon$ satisfying
$$
\|\mu_1^\varepsilon\|_{L^\infty(\mathcal M, d\sigma)}\leq
\frac{1}{|D_\varepsilon^k| \,I_\varepsilon}\leq
\frac{1}{|D_\varepsilon^k| \,(1-C_g\varepsilon)}\,, \quad
\mbox{with } I_\varepsilon=\inf_{x\in\bar D_\varepsilon^k} J_g(x).
$$
where we used \eqref{geodiff}. Therefore, choosing
$m_1=\tfrac{c}2|D_\varepsilon^k|\,(1-C_g\varepsilon)$, then $f(x)>
m_1 \mu_1^\varepsilon$ on $x\in D_\varepsilon^k$, and we can
decompose $\mu$ as a convex combination
$$
\mu= m_1  \mu_1^\varepsilon+ m_2 \mu_2^\varepsilon,
$$
where $\mu_2^\varepsilon\in\prob(\R^N)$.

We are going to send now all mass from $\mu_1^\varepsilon$ to a
Dirac Delta at $x_0=0$. Let us define $\pi:\real^N\longrightarrow
\real^N$ by $\pi\equiv 0$ and $\mu^\varepsilon:=
m_1\pi\#\mu_1^\varepsilon+ m_2 \mu_2^\varepsilon$, then
$\pi\#\mu_1^\varepsilon=\delta_{0}$. Moreover, $\mu^\varepsilon$
is a small perturbation of $\mu$ in $d_\infty$:
\begin{equation}\label{tech2}
d_\infty(\mu,\mu^\varepsilon)\leq \varepsilon(1+\|\nabla
g\|_{L^\infty(D_\varepsilon^k)}).
\end{equation}
To check this just take a map $\mathcal{T}$ in Definition
\ref{disinfty} such that $\mathcal{T}(x)=x$ for all
$x\in\mathcal{M}/P_g(D_\varepsilon^k)$ and such that
$\mathcal{T}(x)=0$ for $x\in P_g(D_\varepsilon^k)$. Thus, the
maximum displacement produced by the transport map $\mathcal{T}$
is bounded by the maximum of $|P_g(x)|$ for $x\in
P_g(D_\varepsilon^k)$ leading to \eqref{tech2} using that $g(0)=0$
and the mean value theorem.

Since $\mu_1^\varepsilon$ has a connected support that contains
$\pi(\supp(\mu_1^\varepsilon))=\{0\}$, we can apply Lemma
\ref{localization} to get:
\begin{align*}
 E[\mu^\varepsilon]-E[\mu]&= m_1^2 \,T[\pi\#\mu_1^\varepsilon,\mu_1^\varepsilon]\\
&=m_1^2
\,T[\pi\#\mu_1^\varepsilon,\nu_{\varepsilon,k}]+m_1^2\left(
T[\pi\#\mu_1^\varepsilon,\mu_1^\varepsilon] -
T[\pi\#\mu_1^\varepsilon,\nu_{\varepsilon,k}]\right)
\end{align*}
Since $\pi\#\mu_1^\varepsilon=\delta_0$, we can use Lemma
\ref{convexity} to estimate the first term, and since moreover
$\mu_1^\varepsilon=P_g\#\nu_{\varepsilon,k}$, we can use Lemma
\ref{errorestimate} to estimate the last two terms, so that we
finally conclude
$$
E[\mu^\varepsilon]-E[\mu]\leq
m_1^2\left[-c_{k,\alpha}\varepsilon^\alpha+C\varepsilon^\alpha\|\nabla
g\|_{L^\infty(D_\varepsilon^k)}\right].
$$
Since $g\in C^1(D_R^k)$ and $\nabla g(0)=0$ imply that $\|\nabla
g\|_{L^\infty(D_\varepsilon^k)}\to 0$ as $\varepsilon\to 0$, thus
if $\varepsilon>0$ is small enough, $E[\mu^\varepsilon]-E[\mu]<0$.

Thus, $\mu^\varepsilon$ is a better competitor in the minimization
of $E$ for $\varepsilon$ arbitrary small. This leads to a
contradiction with the fact that $\mu$ is a local minimizer of $E$
showing Theorem \ref{thm1}.
\end{proof}


\section{Euler-Lagrange approach to study local minimizers in the $d_2$-topology }

So far we have used transport plans to build perturbed measures.
This enabled us to study local minimizers of the interaction
energy with respect to the $d_\infty$-topology. To study local
minimizers with respect to the $d_2$-topology it is actually
possible to use a more classical Euler-Lagrange approach as we
will present in this section. The Euler-Lagrange conditions that
we will derive were formally obtained in \cite{BT2} by perturbing
densities inside and outside their support. Here, we provide a
fully rigorous proof in the case of probability measures endowed
with the distance $d_2$.

\begin{theorem}
Given an interaction potential $W$ satisfying {\rm (H1)--(H2)}.
Let us consider $\mu\in\mathcal{P}_2(\R^N)$ a local minimizer of
$E$ with respect to $d_2$ such that $E[\mu]<\infty$. Then,
\begin{enumerate}
\item[(i)] $(W\ast \mu)(x)=2E[\mu]$ $\mu$-a.e.

\item[(ii)] $(W\ast \mu)(x)\leq 2E[\mu]$ for all $x\in
\supp(\mu)$.

\item[(iii)] $(W\ast\mu)(x)\geq 2E[\mu]$ for a.e. $x\in\R^N$.
\end{enumerate}
\end{theorem}

\begin{proof}
As usual, we assume that $W\geq 0$ without loss of generality.
Lemma~\ref{lsc-conv} implies that $W\ast \mu$ is well defined,
lower semicontinuous, and non-negative.

In order to prove the first two items, let us choose $\varphi\in
C^\infty_0(\R^N)$ to define
\[
   \nu = \left(\varphi-\int_{\R^N} \varphi d\mu\right) \mu:= a(x)\mu
\]
and $\mu_\epsilon = \mu + \epsilon\nu = (1+\epsilon a(x))\mu$ with
$\epsilon>0$ to be specified. It is clear that
$\mu_\epsilon(\real^N) = 1$ since $a(x)$ has zero integral with
respect to $\mu$. Moreover, since $a(x) \geq
-2\|\varphi\|_{L^\infty}$ then $\mu_\epsilon\geq 0$ for $\epsilon
< \frac{1}{2\|\varphi\|_{L^\infty}}=\epsilon_\varphi$. Therefore,
$\mu_\epsilon\in \mathcal{P}(\R^N)$ for all
$\epsilon<\epsilon_\varphi$. It is easy to check that
$\mu_\epsilon\in \mathcal{P}_2(\R^N)$, that $\mu_\epsilon
\rightharpoonup \mu$ weakly-$\ast$ as measures, and
\[
   \int_{\R^N} |x|^2 d\mu_\epsilon \to \int_{\R^N} |x|^2 d\mu \,.
\]
In fact, since $\epsilon a(x)$ converges uniformly to $0$, these
claims follow by dominated convergence theorem. Therefore, we
conclude that
\[
   d_2(\mu_\epsilon,\mu)\to 0 \quad\text{as}\quad \epsilon\to 0.
\]
Note that it is not true that $d_\infty(\mu_\epsilon,\mu)\to 0$ as
$\epsilon\to 0$ since for localized test functions $\varphi$ in
subsets of the $\supp (\mu)$, we are always forced to move mass in
$\supp(\mu)$ for a fixed distance not depending on $\epsilon$.

Now, since $\mu$ is a local minimizer in $d_2$ then
$E[\mu_\epsilon]\geq E[\mu]$ for $\epsilon$ small enough.
Moreover, since $\mu$ has finite energy, then
$E[\mu_\epsilon]<\infty$ and we can expand it as
\[
\frac{E[\mu_\epsilon] - E[\mu]}{\epsilon} = \iint_{\R^N\times
\R^N} W(x-y)d\nu(x)d\mu(y) + \frac{\epsilon}{2}\iint_{\R^N\times
\R^N} W(x-y)d\nu(x)d\nu(y)\geq
   0\, ,
\]
with both integral terms well-defined. As $\epsilon\to 0$, we
easily get
$$
\iint_{\R^N\times \R^N} W(x-y)d\nu(x)d\mu(y)\geq 0
$$
or equivalently,
$$
\int \varphi \left[(W\ast\mu)(x)-2E[\mu]\right]d\mu(x)\geq0
$$
for all $\varphi\in C_0^\infty(\R^N)$. Since one can take either
$\varphi$ or $-\varphi$ as test functions, we deduce
$$
\int \varphi \left[(W\ast\mu)(x)-2E[\mu]\right]d\mu(x) = 0
$$
for all $\varphi\in C_0^\infty(\R^N)$, and thus (i) is satisfied
a.e. $\mu$.

Let us now prove (ii). Take $x\in\supp(\mu)$ then there exists
$\lbrace x_n\rbrace_{n\in\N} \to x$ with $x_n\in \supp(\mu)$, such
that $(W\ast\mu)(x_n)=2E[\mu]$. The existence of such a sequence
is ensured since $\mu(B(x,\epsilon))>0$ for all $\epsilon >0$ by
definition of the support of $\mu$. Then, by lower semicontinuity
of $W\ast \mu$ we get
\[
    (W\ast\mu)(x)\leq \liminf_{n\to\infty}
    (W\ast\mu)(x_n)=2E[\mu]\, .
\]
and then (ii) is satisfied.

In order to show (iii), we consider different variations to the
ones constructed above. Take $\psi\in\C_0^\infty(\R^N)$, $\psi\geq
0$ and then take
$$
\nu = \psi - \left( \int_{\R^N}\psi \,dx\right)\mu.
$$
Again, defining $\mu_\epsilon = \mu+\epsilon\nu$, then it verifies
$\mu_\epsilon (\real^N)=1$ and if $\epsilon < 1/\int \psi \,dx$
then $\mu_\epsilon \geq 0$. Let us remark that this cannot be done
for a changing sign test function $\psi$. As previously, it is
easy to check that
\[
   d_2(\mu_\epsilon,\mu)\to 0\quad\text{as}\quad\epsilon\to 0,
\]
note that it is not true that $d_\infty(\mu_\epsilon,\mu)\to 0$ as
$\epsilon\to 0$ since we always need to transport some mass from
outside the support of $\mu$ to $\supp (\mu)$.

Proceeding similarly as in point (i), we get
\[
   \iint_{\R^N\times \R^N} W(x-y)d\nu(y)d\mu(x) \geq 0
\]
taking $\epsilon \to 0$ in $E[\mu_\epsilon]\geq E[\mu]$.
Therefore, we conclude that
\[
   \int_{\R^N}\left( (W\ast \mu)(x)-2E[\mu]\right) \psi \,dx\geq
   0\, ,
\]
for all $\psi\in C_0^\infty(\R^N)$, $\psi\geq 0$. This readily
implies (iii).
\end{proof}

\begin{remark}
Note that putting together (i), (ii), and (iii) in previous
theorem, we conclude that
\begin{equation*}
\begin{cases}
(W\ast \mu)(x) &= 2E[\mu]\quad \text{for a.e.}\quad x\in\supp(\mu)\\
(W\ast \mu)(x) &\geq 2E[\mu]\quad \text{for a.e.} \quad x\in
\R^N\setminus \supp(\mu).
\end{cases}
\end{equation*}
if $\mu$ is absolutely continuous with respect to the Lebesgue
measure. These two properties are the Euler-Lagrange conditions
that were found for densities in \cite{BT2}.
\end{remark}

\begin{remark}
Let us now clarify the differences between local minimizers in the
$d_2-$topology and local minimizers in the $d_\infty-$topology.
Following \cite{FellnerRaoul1}, let consider as an example the
interaction potential $W(x):=-x^2+\frac{x^4}2$ in one dimension.
Then,
$$
\rho_m=m\delta_0+(1-m)\delta_1
$$
is a critical point of the interaction energy for any $m\in
[0,1]$. Theorem {\rm 3.1} in \cite{FellnerRaoul1} shows that the
measure $\rho_m$ is a local minimizer in the $d_\infty-$topology
as soon as $m\in(1/3,2/3)$. Indeed, what is proven is the stronger
statement that $\rho_m$ is locally asymptotically stable for the
aggregation equation \eqref{pdes1} with respect to any
perturbation in the $d_\infty-$topology. However, $E(\rho_m)=\frac
12\left(m-\frac 12\right)^2-\frac 18$, so that $\rho_{1/2}$ only
can be a local minimizer of the energy in the $d_2-$topology (and
one can prove that it actually is). This shows that the set of
local minimizers with respect to the $d_2-$topology is strictly
contained in the set of local minimizers with respect to the
$d_\infty-$topology. Moreover, numerical simulations suggest that, for $m\in(1/3,2/3)$, $\rho_m$
is actually stable (although not asymptotically stable) with
respect to small $d_2-$perturbations. As a consequence, when using
a gradient flow approach to compute numerically minimizers of the
interaction energy via particles, one obtains $d_\infty-$local
minimizers which typically are not $d_2-$local minimizers (see
e.g. Fig {\rm 2} of \cite{FellnerRaoul1}).
\end{remark}


\section{Numerical Experiments}
In this section we conduct a numerical investigation of the local
minimizers of the discrete interaction energy \eqref{defEdiscrete}
with high number of particles. The gradient flow of
\eqref{defEdiscrete} is given by the system of ODEs:
\begin{equation}\label{eq:pde_particles}
\dot{X}_i=-\sum_{\substack{j=1\\j\neq i}}^{n}m_j\nabla W(X_i-X_j).
\end{equation}
In order to  efficiently find local minimizers of
\eqref{defEdiscrete}, we solve \eqref{eq:pde_particles} by an
explicit Euler scheme with an adaptive time step chosen as the
largest possible such that the discrete energy
\eqref{defEdiscrete} decreases. This scheme is nothing else than a
gradient descent method for the discrete energy
\eqref{defEdiscrete}. Although this method might not be accurate
enough for the dynamics, it is efficient to find local minimizers
of the discrete energy. In stiffer situations an explicit
Runge-Kutta method is used instead. These methods are essentially
the ones proposed in \cite{BUKB,soccerball}.
\begin{table}[ht]
\centering
\begin{tabular}{||m{1.7cm}|m{3cm}|m{3cm}|m{3cm}|m{3cm}||}
\hline
& $\qquad\mbox{Dim = 0}$ & $\qquad\mbox{Dim = 1}$ & $\qquad\mbox{Dim = 2}$ & $\qquad\mbox{Dim = 3}$\\
\hline
&(a)&&&\\
$\alpha=2.5$
&\begin{center}\includegraphics[height=0.1\textheight]{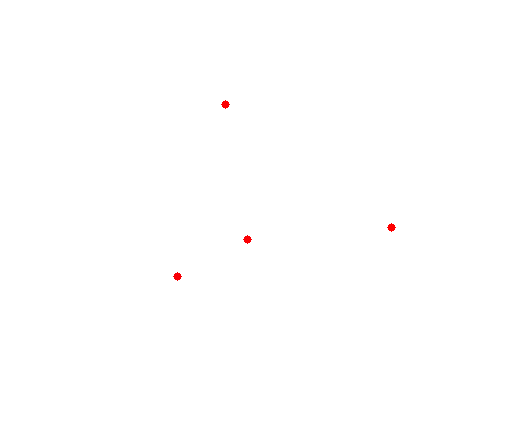}\end{center}
& & & \\
\hline
&&&(b)&(c)\\
$\alpha=1.25$ & &
\begin{center}{\Huge ?}\end{center}
&\begin{center}\includegraphics[height=0.1\textheight]{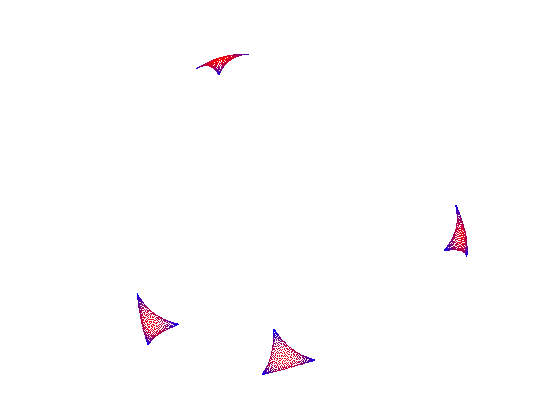}\end{center}
&\begin{center}\includegraphics[height=0.1\textheight]{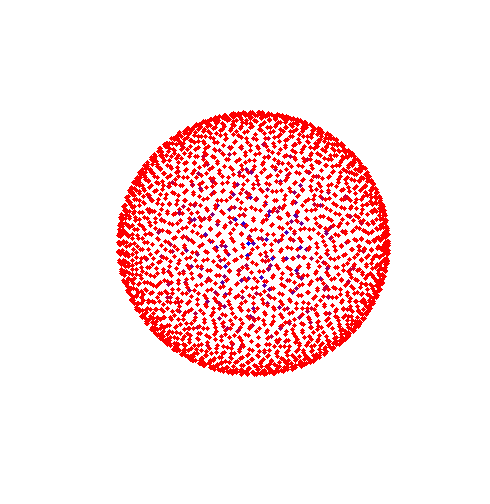}\end{center}
\\ \hline
&&&(d)&(e)\\
$\alpha=0.5$ & &
&\begin{center}\includegraphics[height=0.1\textheight]{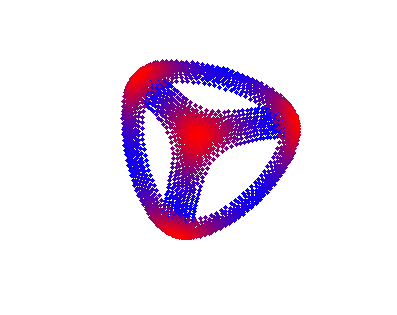}\end{center}
&\begin{center}\includegraphics[height=0.1\textheight]{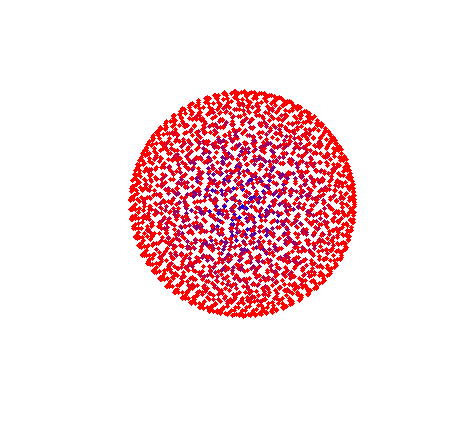}\end{center}
\\ \hline
&&&&(f)\\
$\alpha=-0.5$ & & &
&\begin{center}\includegraphics[height=0.1\textheight]{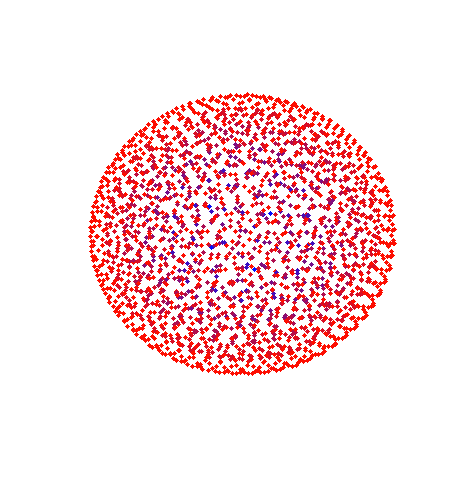}\end{center}
\\ \hline
 \end{tabular}
\vspace{3mm} \caption{Minimizers of $E^{n}_W$ in $\real^3$ for
various power-law potentials with $n=2,500$.} \label{tab:taula4}
\end{table}

The results of these simulations in two dimensions with power-law
potentials were presented in the introduction, see
Table~\ref{figure_intro}. In Subsection \ref{threeD}  we discuss
similar computations in three dimensions. We also provide
numerical experiments suggesting that for some potentials, there
are local minimizers of the interaction energy with mixed
dimensionality, that is, local minimizers that are the sum of measures
whose support have different Hausdorff dimension.

In Subsection \ref{perturbation} we show how our numerical
results can be further understood by using the results from
\cite{KSUB,BUKB,BCLR}, where a careful stability analysis  of ring
solution (in 2D) and spherical shell  solution (in 3D) was
conducted. We also show how this stability analysis connects to
the analytical results presented in this paper.

\subsection{Numerical experiments in 3D} \label{threeD}

We first compute numerically local minimizers of $E^n_W$ where $W$
is the power-law potential defined by \eqref{powerlaw}. Recall
that $\Delta W(x)\sim -1/{|x|^{\beta}}$ with $\beta=2-\alpha$ as
$x\to 0$. The computations are performed with $n=2,500$ particles.
The results are shown in Table~\ref{tab:taula4} and are discussed
below:

\begin{itemize}
\item Subfigure (a): $\alpha=2.5$ and $\gamma=5$. The support of
the minimizer has zero Hausdorff dimension in agreement with
Theorem \ref{thm1}. Actually, in this particular case it is
supported on 4 points forming a tetrahedron.

\item Subfigure (b) and (c): the two potentials have the same
behavior at the origin, $\alpha=1.25$, but different attractive
long range behavior ($\gamma=15$ and $\gamma=1.4$ respectively).
Theorem \ref{dimension} shows that the Hausdorff dimension of the
support must be greater or equal to $\beta=2-\alpha=0.75$.
Numerically, we observe that the local minimizer for the first
example has a two-dimensional support and the minimizer for the
second example has a three-dimensional support. We did not choose
the value $\alpha=1.5$ because we were not able to obtain
 a change of dimensionality of the stable steady states
varying $\gamma>\alpha$. Note that $\alpha=1.5$ is always above the
instability curve for radial perturbations which meets line
$\alpha=\gamma$ at the point $(\sqrt{2},\sqrt{2})$. See Figure
\ref{fig:diagramShellsrevisited} and Subsection 6.2 for more
details.

\item Subfigure (d) and (e):  the two potentials have the same
behavior at the origin, $\alpha=0.5$, but different attractive
long range behavior ($\gamma=23$ and $\gamma=1.4$ respectively).
Theorem \ref{dimension} shows that the Hausdorff dimension of the
support must be greater or equal to $\beta=2-\alpha=1.5$. Numerically,
we observe that the local minimizer for the first
example has a two-dimensional support and the local minimizer for
the second example has a three-dimensional support.

\begin{figure}[ht]
{\subfigure{\includegraphics[scale=0.3]{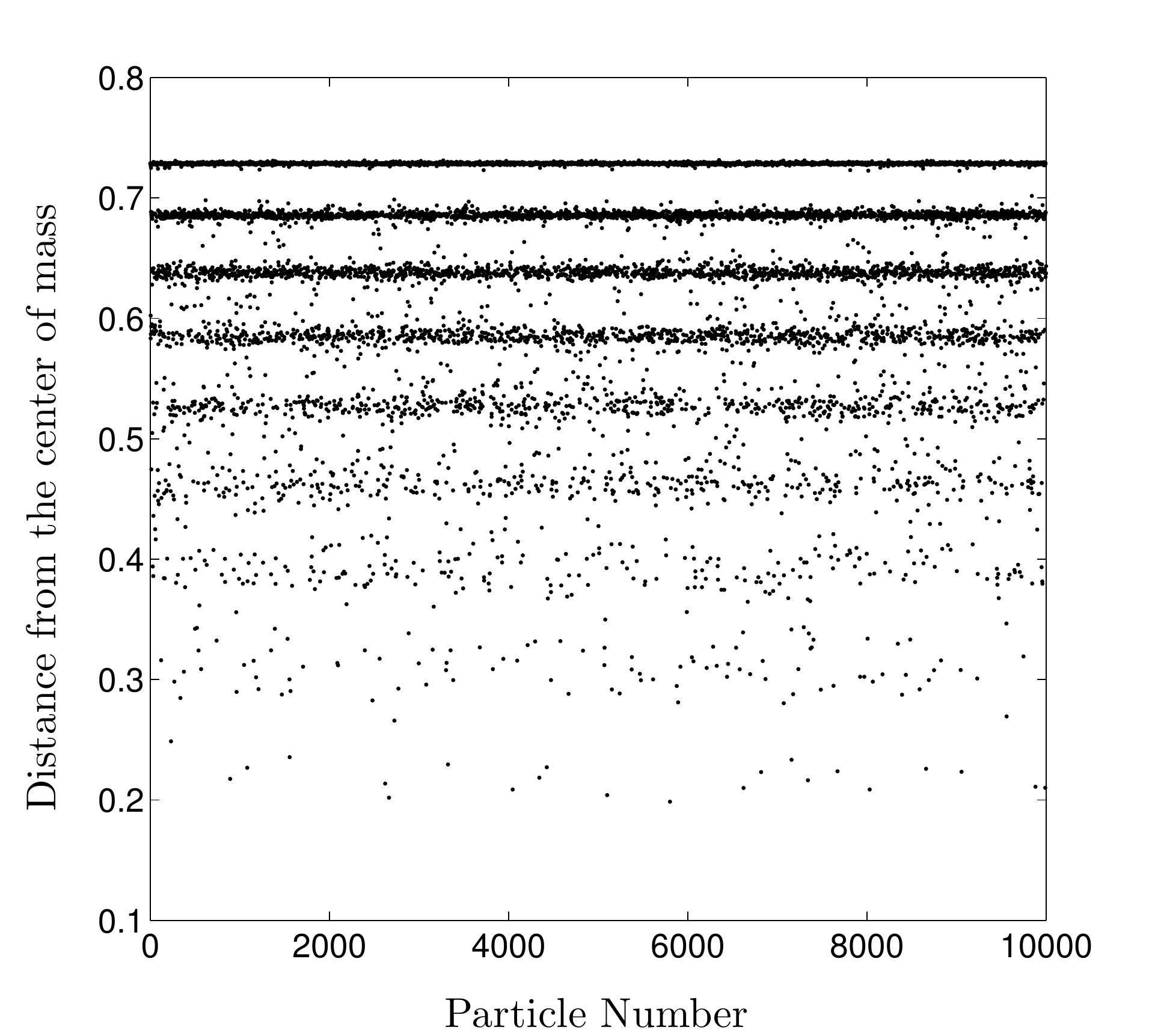}}}
\caption{Distances of the particles from the center of mass for
the power law potential with $\alpha=-0.5$, $\gamma=5$ in 3D. Case
(f) in Subsection 6.1 in Table~\ref{tab:taula4} with $n=10,000$.}
\label{fig:radiusofparticles}
\end{figure}

\item Subfigure (f): $\alpha=-0.5$ and $\gamma=5$. Theorem
\ref{dimension} proves that the Hausdorff dimension of the support
must be greater than $\beta=2-\alpha=2.5$, which can also be
observed numerically. In Figure \ref{fig:radiusofparticles}, we
have represented the radius of particles to the center of mass.
The particles seem to organize into successive two dimensional
layers. Such lattices were also observed in \cite{KHP}, and it is
related to the finite number of particles used in the simulations.
\end{itemize}

Notice that we were not able to find examples of interaction
potentials leading numerically to a local minimizer with one
dimensional support. We could however observe such situations with
an additional asymmetric confining potentials, we thus believe it
should be possible to produce such cases.

A natural question following Tables~\ref{figure_intro}
and~\ref{tab:taula4} is whether it is possible to produce local
minimizers that are a sum of two measures whose support have
different Hausdorff dimensions. A possible candidate was already
observed in \cite{BUKB}. Here, we analyze it more carefully with
much larger number of particles. From our simulations, it seems
that the interaction potential $W(x)=w(|x|)$ with $w$ defined by
\[
   -w'(r)=\tanh((1-r)a)+b,\quad  a=5, \quad  b=0.5,
\]
leads numerically to a local minimizer consisting in a ball
(Hausdorff dimension three) inside a spherical shell (Hausdorff
dimension two), see Figure \ref{fig:3d+2d}.

\begin{figure}[ht!]=
\hspace{-1cm}\includegraphics[height=0.35\textheight]{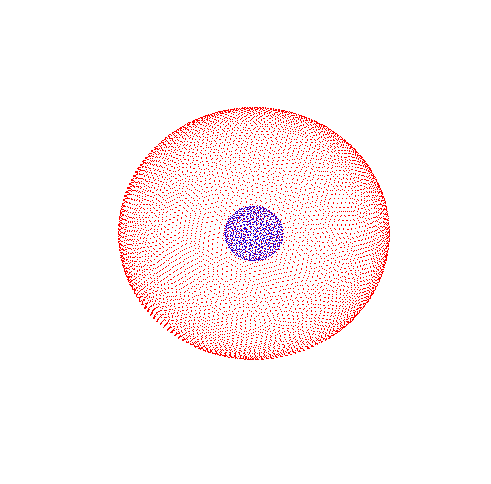}\hspace{-7mm}
\includegraphics[height=0.30\textheight]{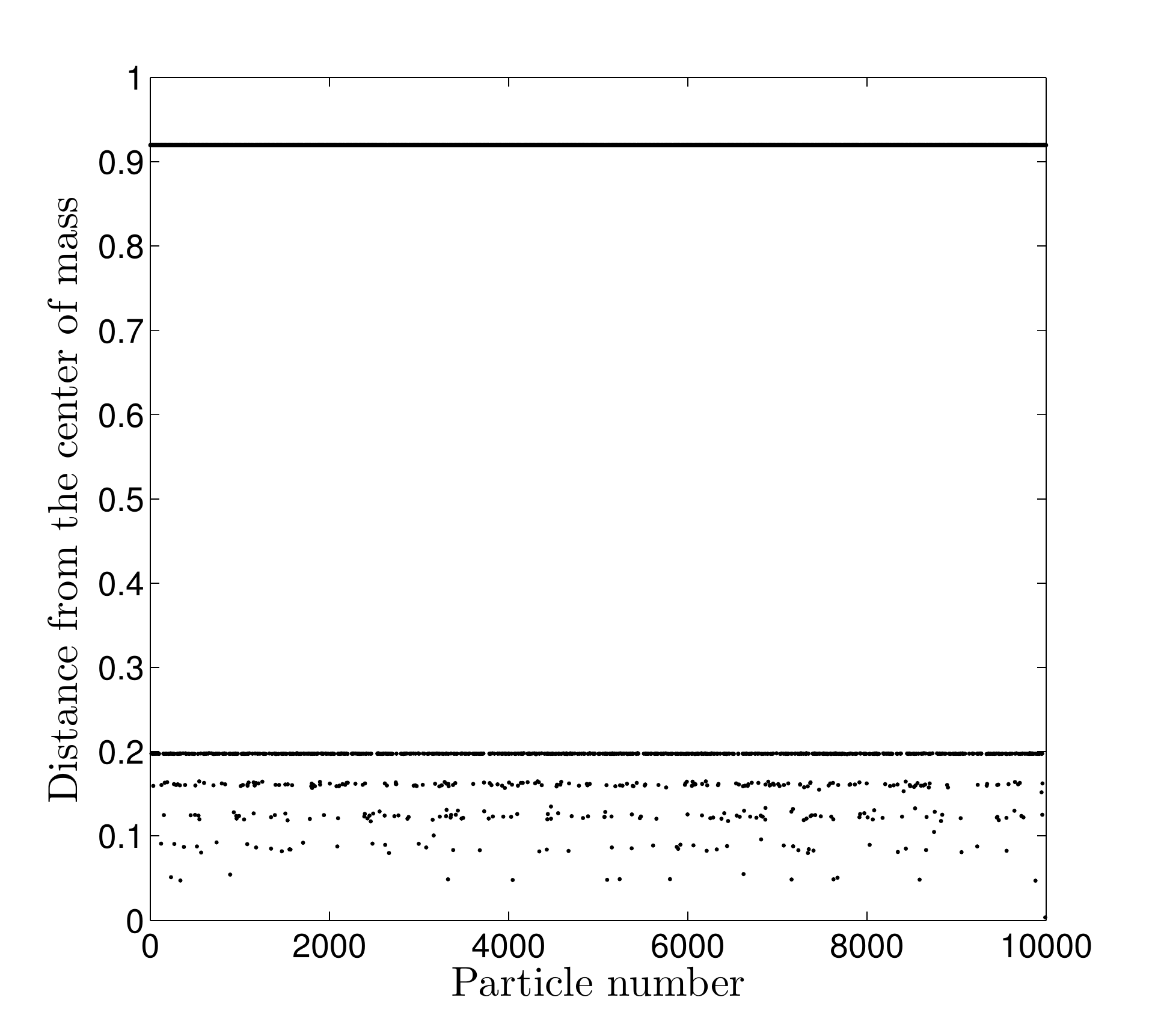}
\caption{Left: Local minimizer in 3D with $n=10,000$. Right:
Distance of the particles from the center of mass.}
\label{fig:3d+2d}
\end{figure}
The distance of each particle to the center of mass is displayed
on the right part of Figure \ref{fig:3d+2d}. The inner ball
appears to be composed of five equally spaced layers of particles.
This is most likely due to the fact that particles are organized
into a lattice configuration, and therefore the distances between
the particles and the origin do not form a continuum. It is
instructive to compare the distribution of the radius of the particles   in the
right subplot of Figure \ref{fig:3d+2d} with the one in Figure
\ref{fig:radiusofparticles} for the case of an approximated local
minimizer with three dimensional support, i.e., Case (f) of
Table~\ref{tab:taula4}. Although Theorem \ref{dimension}
guarantees that the support of the local minimizer corresponding
to Figure \ref{fig:radiusofparticles} has Hausdorff dimension
greater or equal to $2.5$, we can also observe that particles
arrange themselves in layers. Notice that in dimension N=2, such
artifacts also appear in simulations using a finite number of
particles, see Figure 4 in \cite{KHP}.

\subsection{Relationship with previous works on ring and shell solutions}
\label{perturbation}

An important characteristic of the analysis performed in the main
theorems of this work is that we do not assume a specific shape on
the local minimizers. If on the contrary, one is interested by the
special case of delta ring minimizers (in 2D), or spherical shell
minimizers (in 3D), perturbative methods provide more detailed
results.

In \cite{KSUB} the local stability of discrete ring solutions, made of
$N$-particle equally distributed in a circle, was studied for the
$N$-particle system \eqref{eq:pde_particles}. The authors
considered the power law interaction potentials \eqref{powerlaw},
and led a formal linear stability analysis for the continuum ring
solution as steady state of the aggregation equation \eqref{pdes1}
by taking $N\to\infty$. Those predictions were then confirmed
numerically. They could not obtain nonlinear stability of the ring
solution particularly because there is no spectral gap as
$N\to\infty$, i.e, the largest negative eigenvalue tends to 0 when
$N\to\infty$. In \cite{BCLR}, the nonlinear stability of the ring
solutions was proved for radial perturbations, corroborating some
of the formal results of \cite{KSUB}, together with the
instability due to fattening in the complementary set of
parameters.

\begin{figure}[ht]
\includegraphics[scale=0.35]{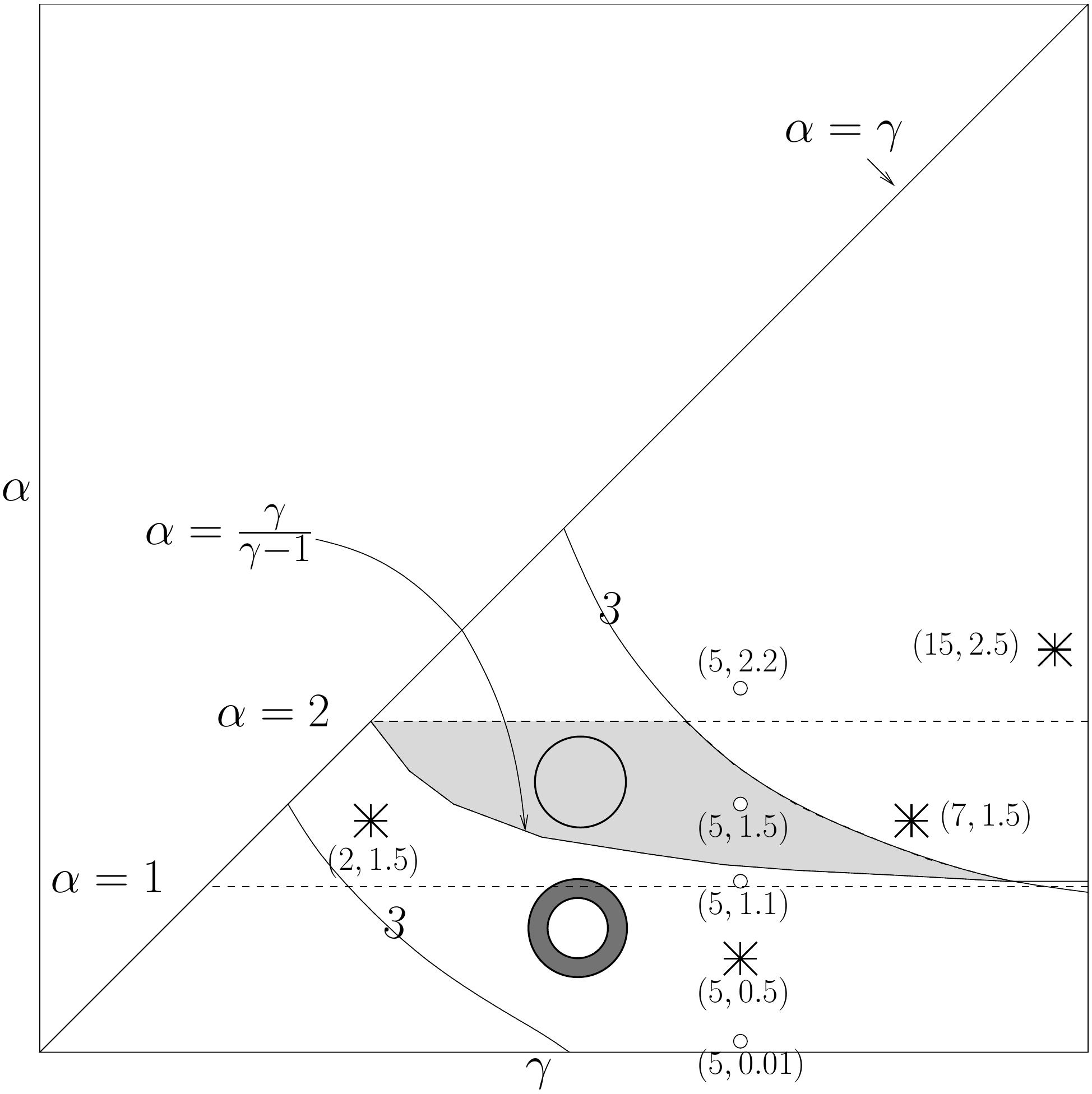}
\caption{Sketch of all the computed cases in dimension $N=2$. The
parameters $(\gamma,\alpha)$ used in Table~\ref{figure_intro} are
marked with $\ast$, while those used in Table~\ref{tab:taula2bis}
are marked with $\circ$. Notice that $\alpha<\gamma$ is necessary
for the interaction potential to be confining. In dark gray is
represented the set of parameters such that a delta ring could be
a local minimizer.} \label{fig:diagramDeltarevisited}
\end{figure}

We have represented this set of parameters in
Figure~\ref{fig:diagramDeltarevisited}, as well as all the
parameters used in the two dimensional numerical simulations of
this article (Tables~\ref{figure_intro} and \ref{tab:taula2bis}).
As the caricature presented in Table~\ref{tab:taula2bis} shows,
crossing the lower border of this set, curve
$\alpha=\gamma/(\gamma-1)$, leads to a ``fattening'' of the delta
ring, that is to minimizers with dimensionality $2$, see
\cite{KSUB,BCLR}. On the other hand, crossing its upper border,
given by the curve marked with 3, does not modify the
dimensionality of the stable steady states as long as $\alpha<2$
(they remain one dimensional), but leads to a ``shape''
instability towards a triangular configuration that breaks the
ring into 3 connected one dimensional components as in case (b) of
Table~\ref{figure_intro}.

Finally, if $\alpha>2$, local minimizers become of dimensionality
$0$, as predicted by Theorem~\ref{thm1}, whereas if $\alpha<1$,
all the minimizers are of dimensionality $2$, as shown by
Theorem~\ref{dimension}.
\begin{table}[ht]
\centering
\begin{tabular}{|c|c|c|c|}
\hline
 \includegraphics[height=0.09\textheight]{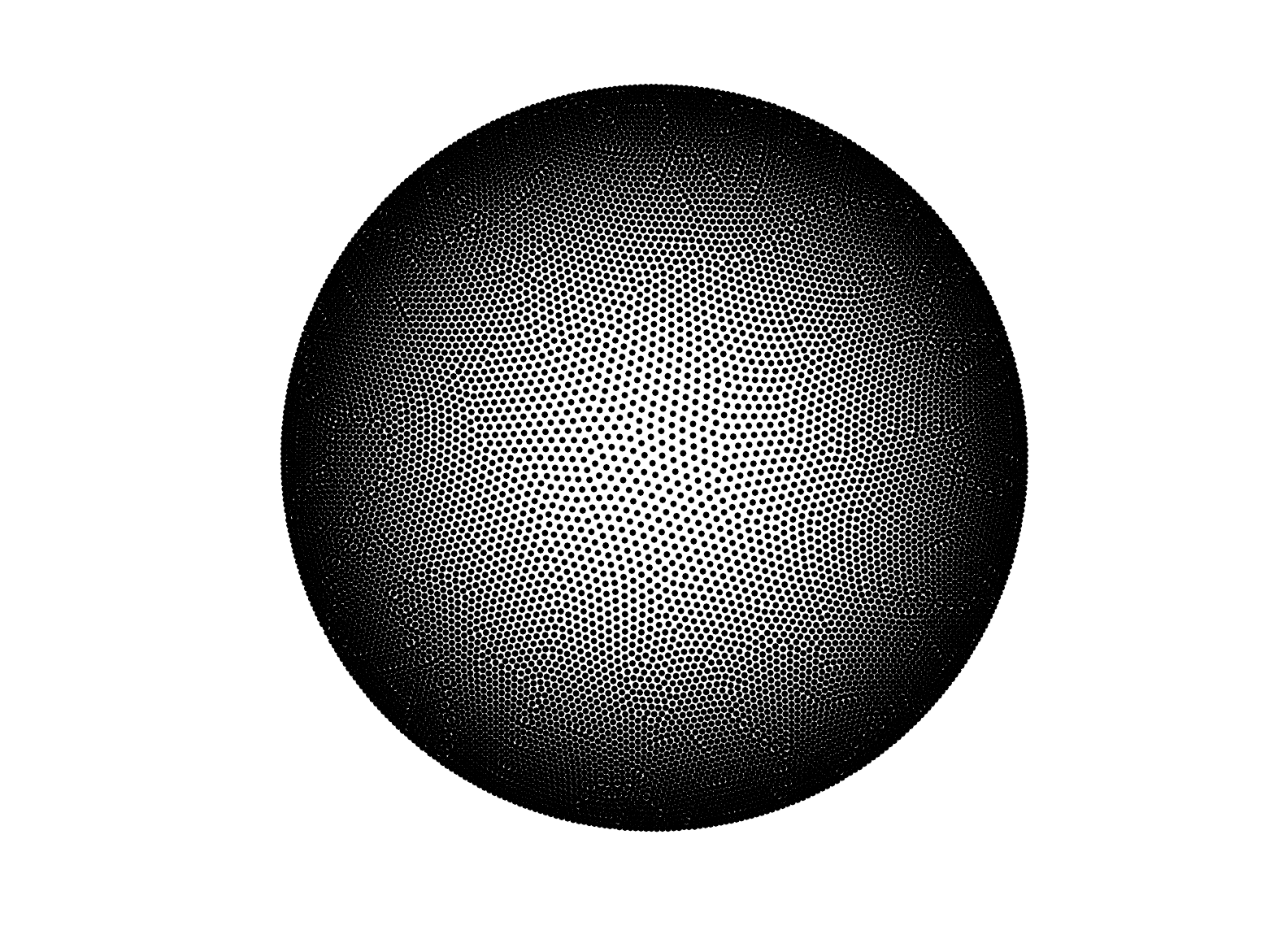}&
 \includegraphics[height=0.09\textheight]{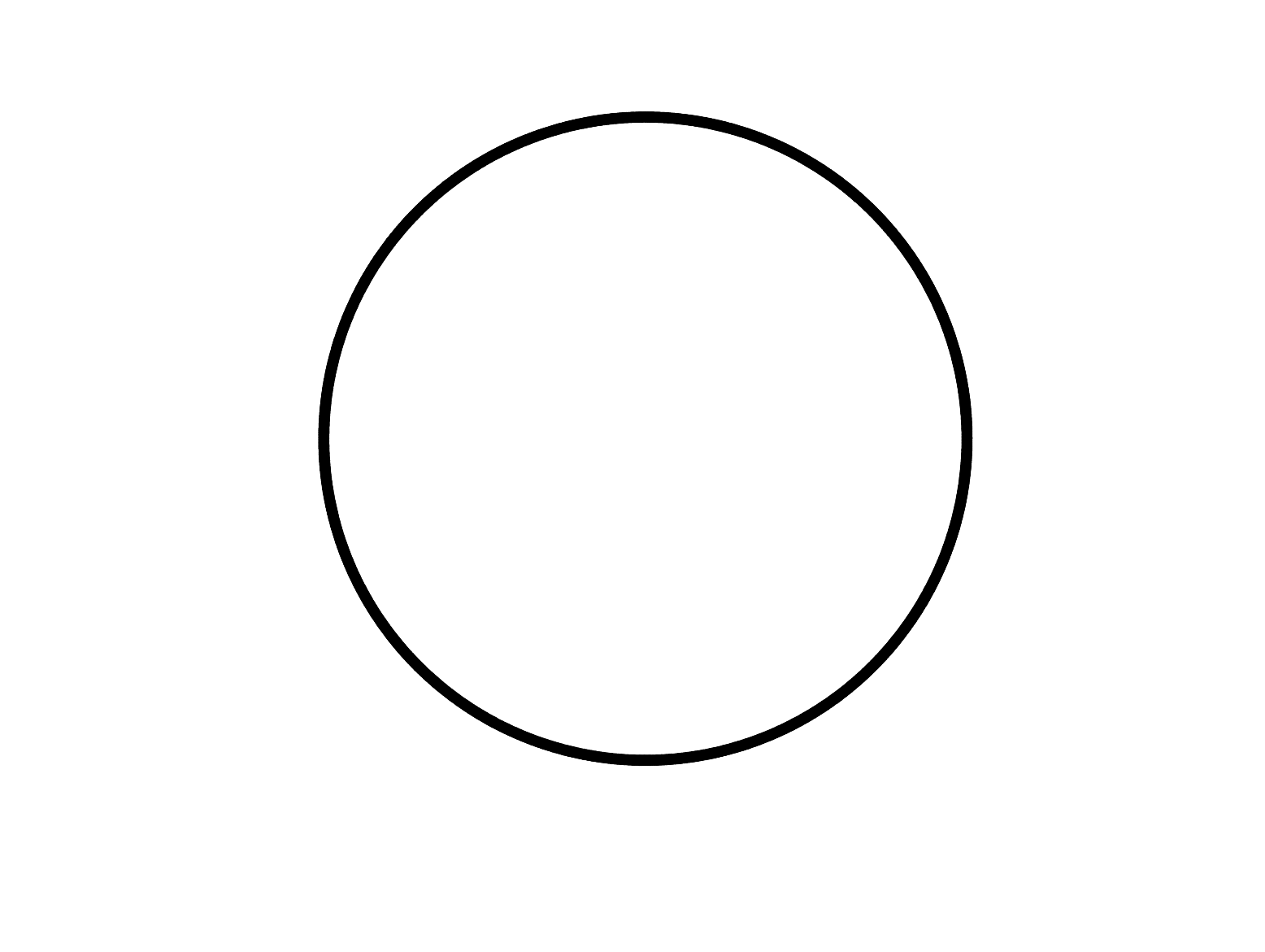}&
 \includegraphics[height=0.09\textheight]{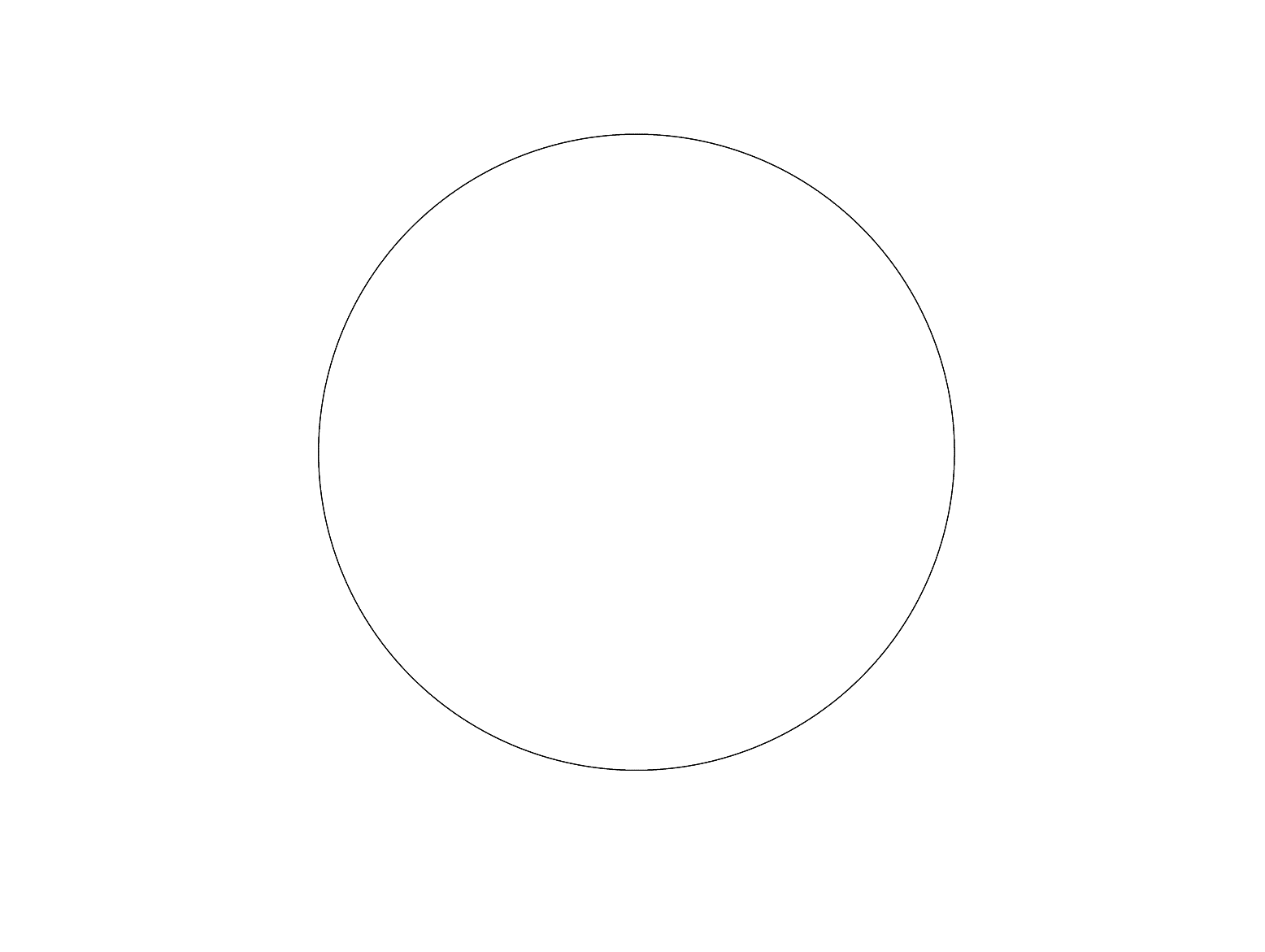}&
 \includegraphics[height=0.09\textheight]{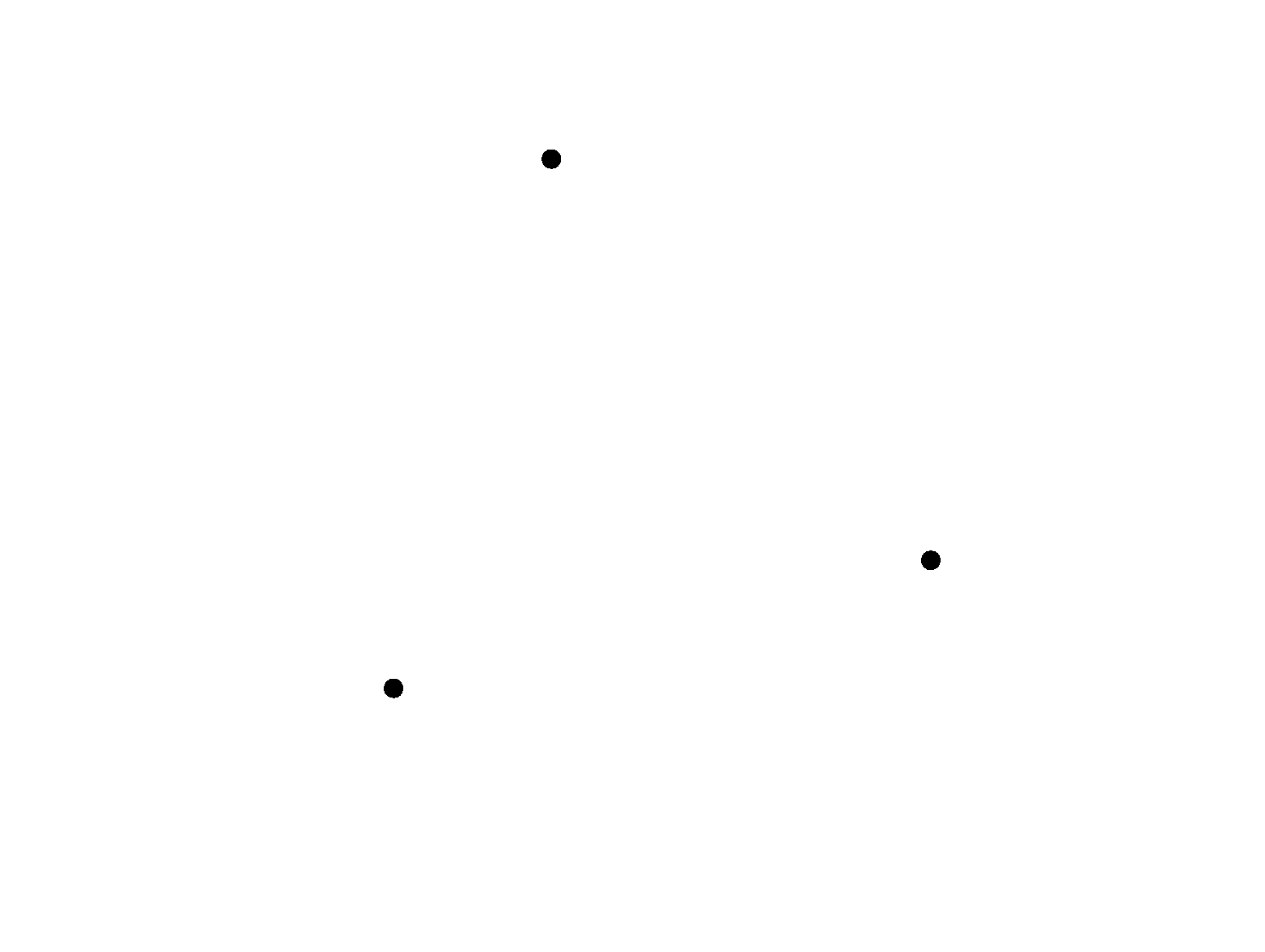}\\\hline
$\alpha=0.01$ & $\alpha=1.1$& $\alpha=1.5$& $\alpha=2.2$ \\\hline
\end{tabular}
\vspace{3mm} \caption{ Evolution of local minimizers when
$\alpha>0$ increases, while $\gamma=5$ remains constant.  The
computations were done with $n=10,000$ particles.}
\label{tab:taula2bis}
\end{table}

In three dimensions, a linear stability analysis of discrete
spherical shell solutions is also possible but it leads to more
cumbersome instability curves, see \cite{BUKB,soccerball}. Again,
the results in \cite{BCLR} give the ``fattening'' instability
curve dividing instability from stability under radial
perturbations. In Figure~\ref{fig:diagramShellsrevisited}, we have
only represented the set of parameters such that the spherical
shells are not local minimizers for spherically symmetric
perturbations, as well as all the parameters $(\gamma,\alpha)$
used for 3D numerical simulations in this article in
Table~\ref{tab:taula4}. Just as we have observed in the 2D case,
crossing the lower border of this set leads to a ``fattening"
instability of the spherical shell.

\begin{figure}[ht]
\includegraphics[scale=0.35]{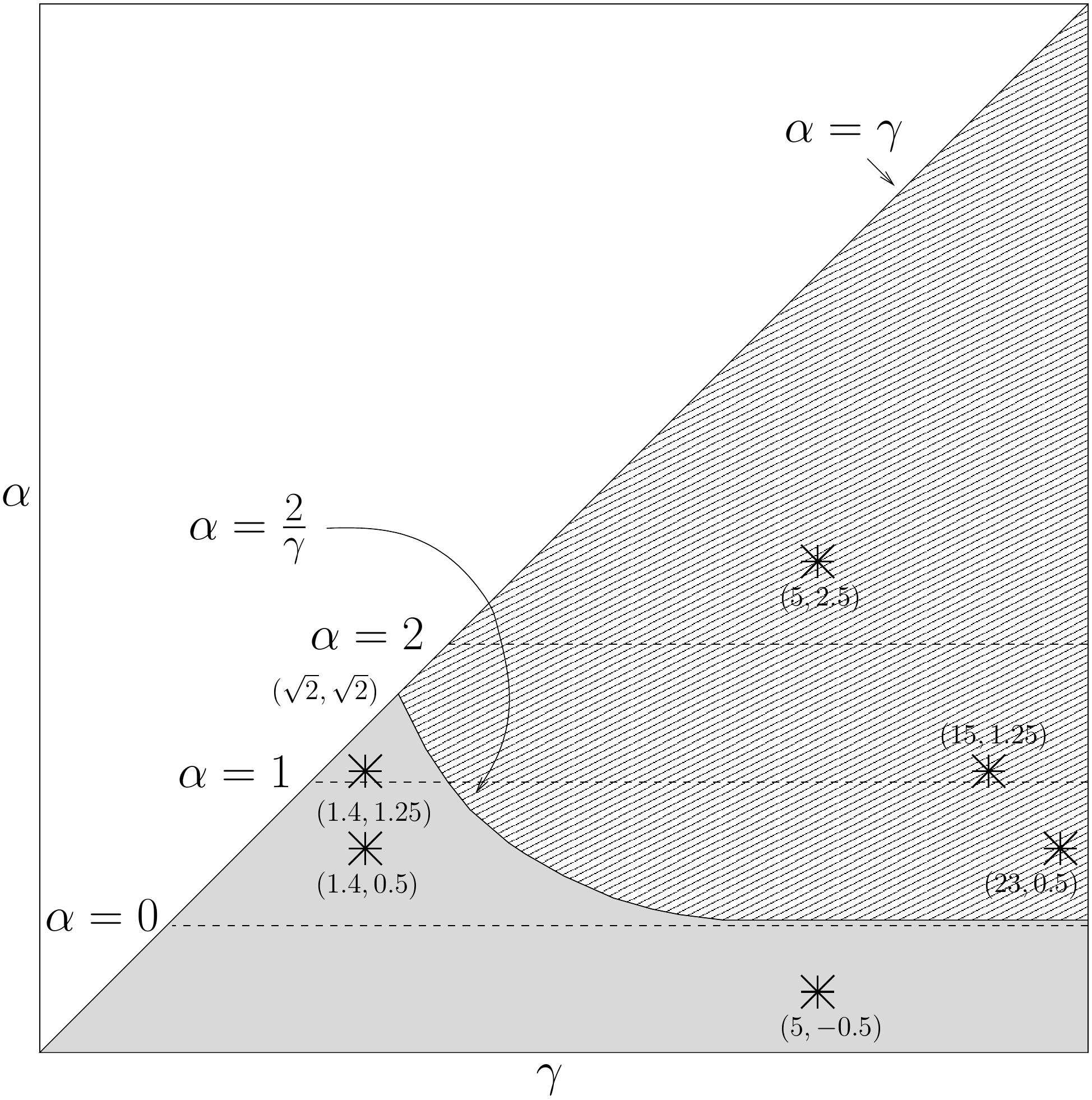}
\caption{ Sketch of all the computed cases in dimension $N=3$. The
parameters $(\gamma,\alpha)$ used in Table~\ref{tab:taula4} are
marked with $\ast$. Notice that $\alpha<\gamma$ is necessary for
the interaction potential to be confining. The curve is the limit
between parameters leading to spherical shell solutions for radial
perturbations (above the curve) and to minimizers of
dimensionality $3$ (below the curve).}
\label{fig:diagramShellsrevisited}
\end{figure}

Notice finally that it is also possible to modify the
dimensionality of the local minimizers with other perturbations of
power law potentials. As an example, in Table~\ref{tab:taula3}, we
consider the following perturbations of the power law potential
\eqref{powerlaw}:
\begin{equation}\label{powerperturb}
   W(x)=-\frac{|x|^\alpha}{\alpha}+\frac{|x|^\gamma}{\gamma}+\frac{3}{2p}\cos(px), \quad\alpha<\gamma,\quad p=3,5.
\end{equation}

In Table~\ref{tab:taula3}
we have represented the power-law case in the
first column, and the perturbations in the next two.  For
$(\gamma,\alpha)=(2,1.5)$, the unperturbed power-law potential
leads to a local minimizer with Hausdorff dimension two. When we
add the perturbation $p=3$,
the dimension of the
minimizer changes to one. Notice that the perturbation does not
alter the local behavior of the potential at the origin or at
infinity, suggesting that Theorem~\ref{dimension} is probably
sharp at least in terms of natural dimensions.

\begin{table}[ht]
\centering
\begin{tabular}{||m{3cm}|m{4cm}|m{4cm}|m{4cm}||}
\hline
& $\qquad\qquad\mbox{Powers}$ & $\qquad\qquad p=3$ & $\qquad\qquad p=5$\\
\hline $(\gamma,\alpha)=(2,1.5)$ &
\begin{center}\scalebox{0.25}{\includegraphics{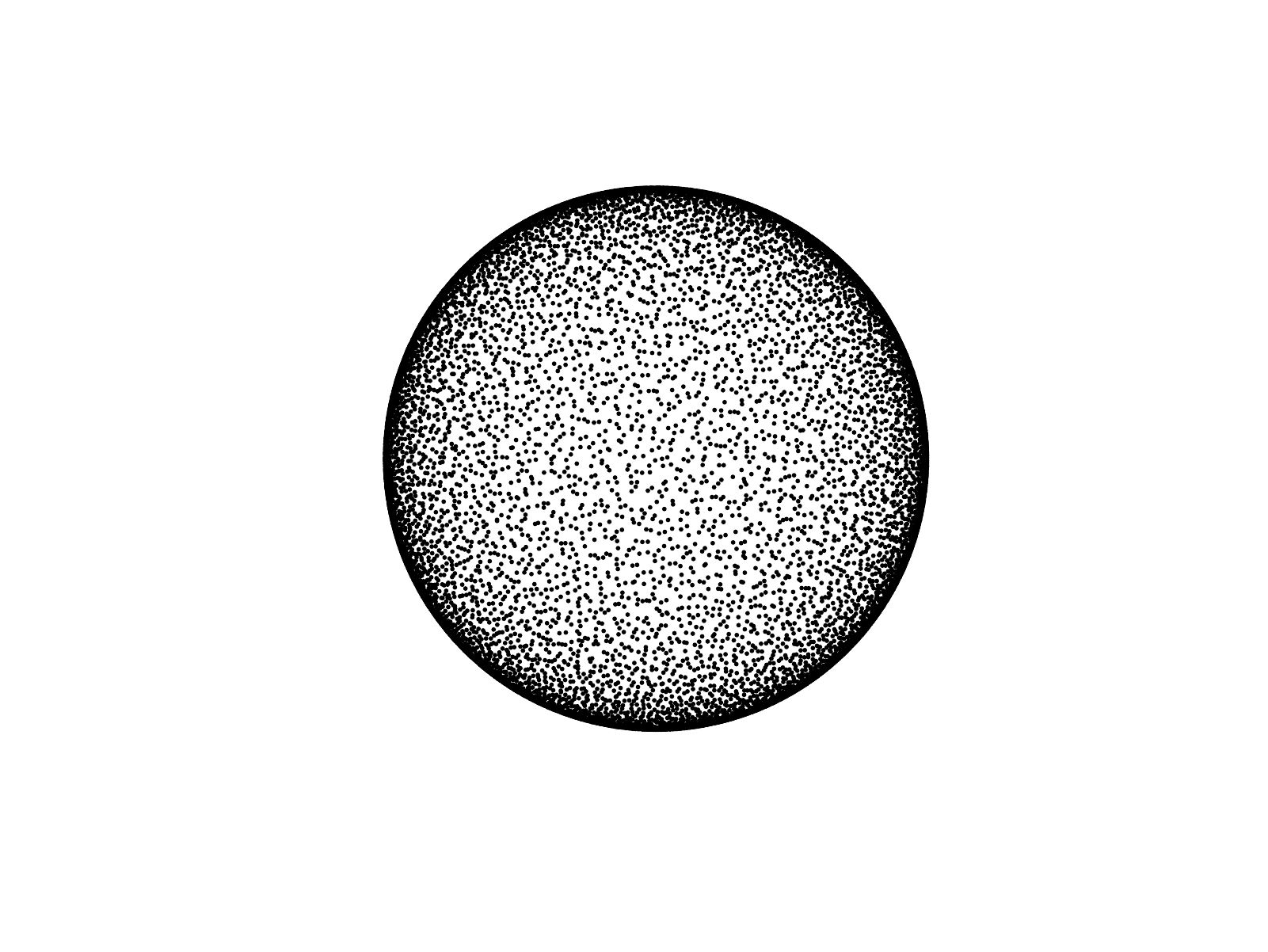}}\end{center}
&
\begin{center}\scalebox{0.25}{\includegraphics{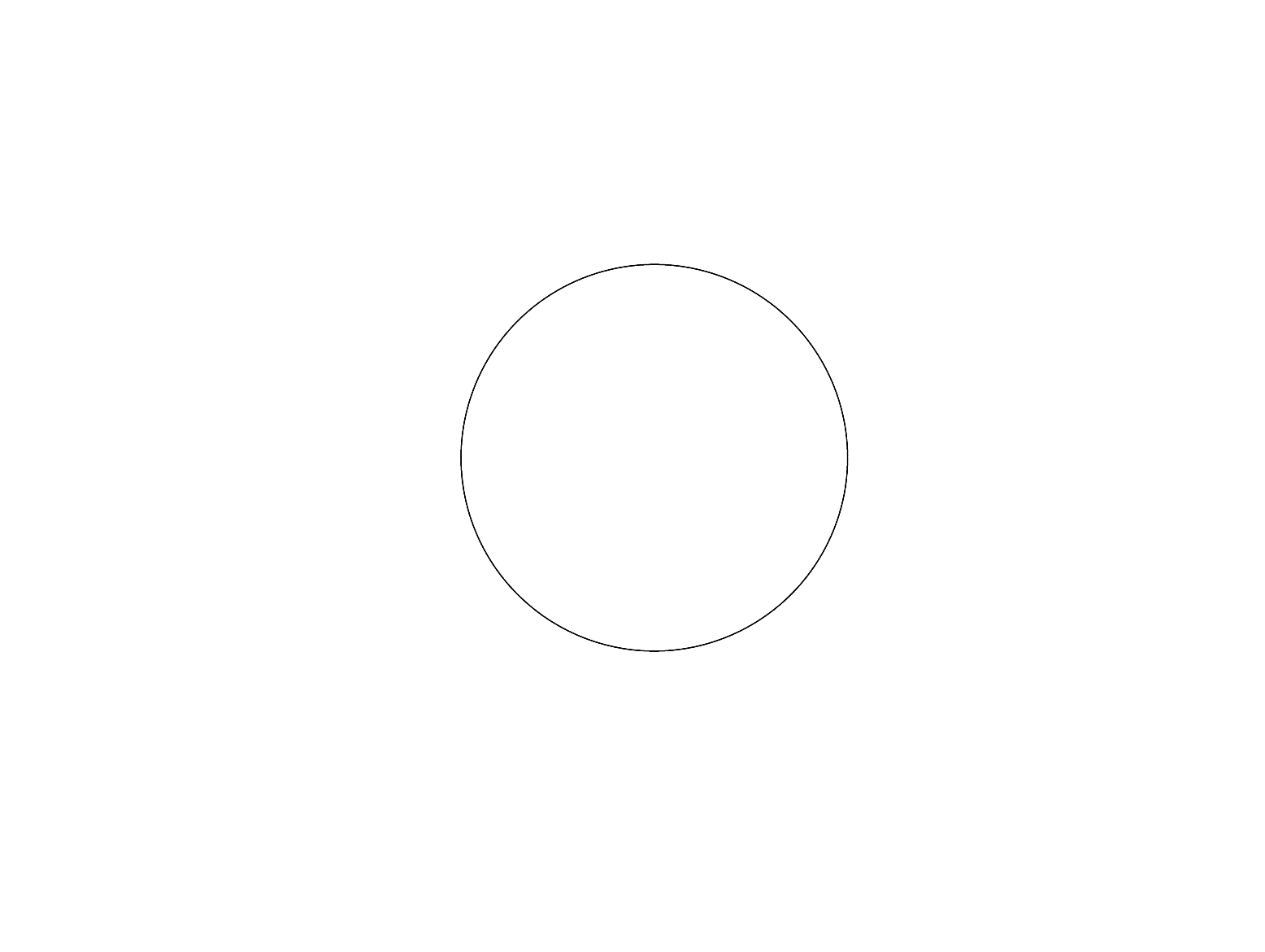}}\end{center}
& \begin{center}\scalebox{0.25}{\includegraphics{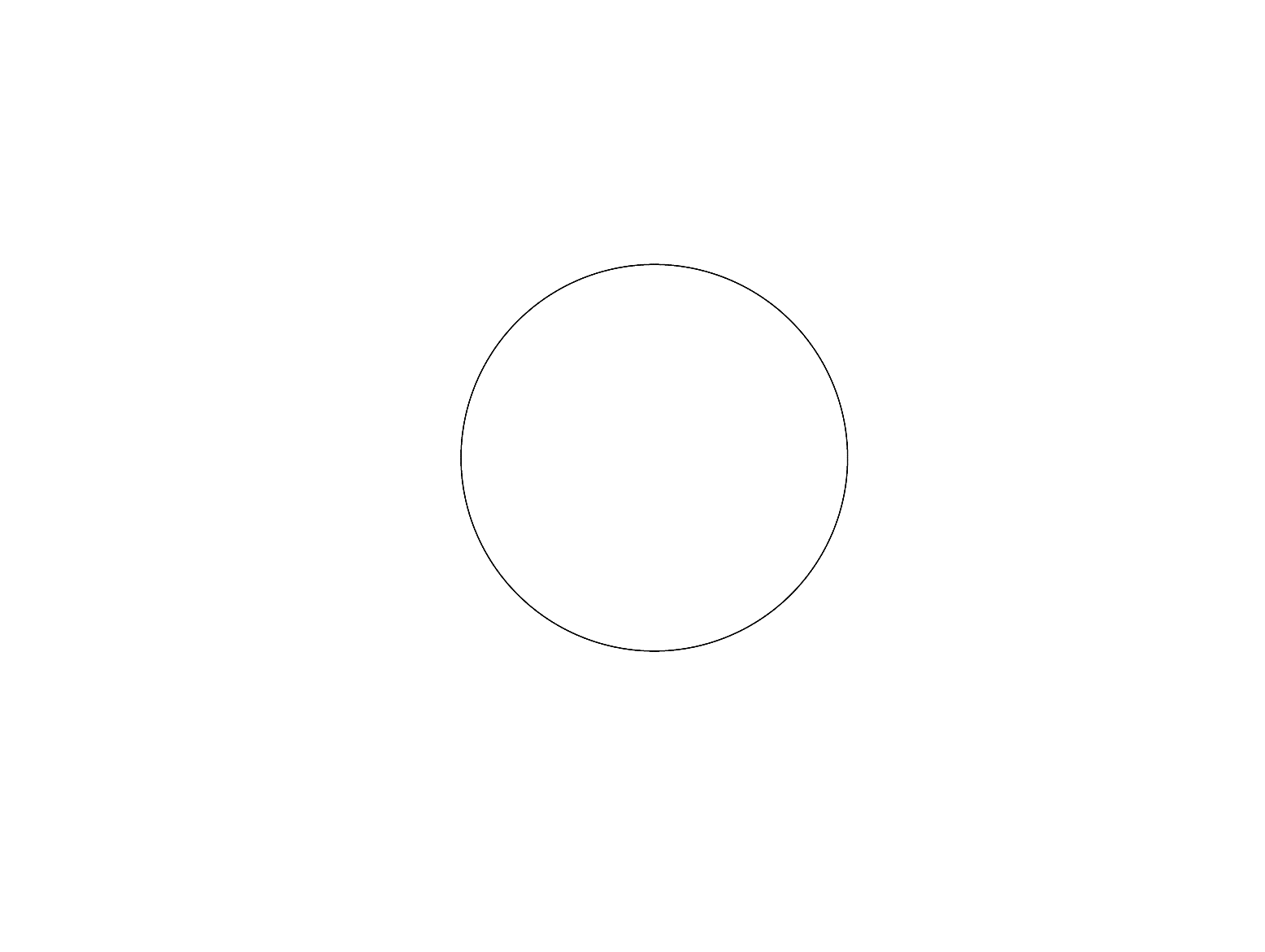}}\end{center} \\
\hline
\end{tabular}
\vspace{3mm}
\caption{Local minimizers with the power-law potential \eqref{powerlaw} and the perturbed potential \eqref{powerperturb}, $n=10,000$.}
\label{tab:taula3}
\end{table}


\subsection*{Acknowledgments}
We would like to thanks James von Brecht for multiple insightful
conversation concerning the numerical section of this paper. DB
and JAC were supported by the projects  Ministerio de Ciencia e
Innovaci\'on MTM2011-27739-C04-02 and 2009-SGR-345 from Ag\`encia
de Gesti\'o d'Ajuts Universitaris i de Recerca-Generalitat de
Catalunya. JAC acknowledges support from the Royal Society through
a Wolfson Research Merit Award. This work was supported by
Engineering and Physical Sciences Research Council grant number
EP/K008404/1. GR was supported by Award No. KUK-I1- 007-43 of
Peter A. Markowich, made by King Abdullah University of Science
and Technology (KAUST). DB, JAC and GR acknowledge partial support
from CBDif-Fr ANR-08-BLAN-0333-01 project. TL acknowledges the
support from NSF Grant DMS-1109805.

\bibliographystyle{plain}
\bibliography{refs}

\end{document}